\documentclass[12pt]{article}
\usepackage{graphicx} 

\usepackage{amssymb,latexsym,amsmath,amsthm,amscd,amsrefs}
\usepackage{stmaryrd}
\usepackage{setspace}
\usepackage{array}
\usepackage{float}
\usepackage[all]{xy} \xyoption{arc}
\usepackage[left=3cm,top=2cm,right=3cm,bottom = 2cm]{geometry}
\usepackage{graphicx}
\usepackage[usenames,dvipsnames]{xcolor}
\usepackage{subfig}
\usepackage{charter}
\usepackage{hyperref}
\usepackage{authblk}
\usepackage{bbding}
\usepackage{blindtext}
\usepackage{multicol}
\usepackage{comment}

\theoremstyle{plain}
\newtheorem{thm}{Theorem}
\newtheorem{lem}[thm]{Lemma}
\newtheorem{prop}[thm]{Proposition}
\newtheorem{cor}[thm]{Corollary}

\theoremstyle{definition}
\newtheorem{dfn}[thm]{Definition}

\theoremstyle{remark}
\newtheorem{rmk}[thm]{Remark}

\newcommand{\ZZ}{\mathbf{Z}}

\newcommand{\CC}{\mathbf C}

\newcommand{\SLZ}{\mathrm{SL}_2(\ZZ)}

\newcommand{\frakg}{\mathfrak g}

\newcommand{\ch}{\operatorname{ch}}
\newcommand{\rk}{\operatorname{rk}}

\title{Balanced root systems and a Schellekens-type  list for holomorphic vertex operator algebras of central charge $32$}

\author[1]{Maneesha Ampagouni} 
\author[1]{Geoffrey Mason}
\author[2]{Michael H.\ Mertens}
\affil[1]{University of California at Santa Cruz}

\affil[2]{RWTH Aachen University}

\date{}

\begin{document}

\maketitle

\begin{abstract}
    We study a special class of  holomorphic vertex operator algebras (VOAs) that we call \emph{balanced}.\ For a balanced, holomorphic VOA $V=\CC\mathbf{1}\oplus V_1\oplus\hdots$ with $c=32$ or $40$ we show that the Virasoro vectors of $V$ and the subVOA generated by $V_1$ coincide and use this result to provide a Schellekens-type list of possible root systems that may occur.
\end{abstract}

\renewcommand{\thefootnote}{\fnsymbol{footnote}} 
\footnotetext{\emph{MSC Subject Classification (2020):} 17B69, 17B22, 11F22}    

\renewcommand{\thefootnote}{\arabic{footnote}} 
\section{Introduction}\label{SIntro}
It has been known since the earliest days of the study of vertex operator algebras (VOAs) that an even unimodular lattice $L$ of rank $c$ gives rise to a lattice VOA $V_L$ of central charge $c$ with especially nice properties (strongly regular and holomorphic) \cites{D,LL,M}, which from now on we abbreviate by simply saying that $V$ is holomorphic.\ However, there are many holomorphic VOAs which do not arise as lattice VOAs, begging the question of the classification of holomorphic VOAs.\ It is well-known that there is exactly one isometry class of even unimodular lattices of rank $8$ \cites{Korkin} and exactly $2$ classes in rank\footnote{The rank of an even unimodular lattice must be divisible by $8$.} $16$ \cites{Witt}, and it turns out that the corresponding lattice VOAs are the only examples of holomorphic VOAs for these values of $c$ \cites{DM1}.\ In rank $24$, there are $24$ even unimodular lattices as discovered by Niemeier
(\cites{N,V}, see also \cite{CS}*{Chapter 18}) but there are additional  holomorphic VOAs of central charge $24$.\ Among these are some that we call \textit{DGM orbifolds}; they arise as certain canonical $\ZZ_2$-orbifolds of lattice theories, one of these being the illustrious Moonshine module \cites{FLM}.\ See Section \ref{SDGM} for further background on this class of VOAs, which were first described in general by \cites{DGM}.\ Schellekens in \cites{S} gave a list of $70$ root systems\footnote{This includes the empty root system.} associated to holomorphic VOAs with $c=24$ (see below for further discussion of this point) and it is now known that these correspond bijectively to holomorphic VOAs $V$ with $c=24$ and $V_1\neq0$.\ For an overview see \cites{LS} and for a different perspective \cites{MoellerScheithauer1, MoellerScheithauer2}. 

 \medskip
Both Niemeier and Schellekens made use of a certain 
\textit{root system} naturally associated to the lattice or VOA respectively.\ This idea is well-known.\ In the case of even lattices, the vectors of norm $2$ form a root system (possibly empty).\ The root system associated to a holomorphic VOA requires more background to explain and we refer the reader to Section \ref{SBR} for further details.

\medskip
Niemeier observed a certain symmetry in his root systems, namely that the dual Coxeter number $h^{\vee}$ of each simple component was a \textit{constant} that did not depend on the root system \textit{per se}.\ Although the root systems associated to VOAs are more complicated in that the simple components come equipped with \textit{levels} (certain positive integers) Schellekens observed a symmetry of his root systems similar to that of Niemeier, namely that
$h^{\vee}/k$ is a constant, where $k$ is the level of the simple component.\ We will refer to this type of symmetry as \textit{balance}.

\medskip
In the present paper we begin by introducing some extensions of the notion of a root system adapted to the needs of VOA theory.\ Our extended root systems encode the data of a \textit{reductive} Lie algebra and a choice of levels, whereas in practice Niemeier and Schellekens needed only standard root systems of semisimple Lie algebras.\ This data is represented by a symbol (still called a root system)
\begin{eqnarray}\label{extRS}
O^f\Phi_{1, k_1}...\Phi_{n, k_n}
\end{eqnarray}
where $k_i$ is the level.\ Similar notation can be found in the lattice theory literature (e.g., \cites{CS}) but where all levels $k_i$ satisfy $k_i=1$.\ Indeed, something even more general that we call a \textit{supplemented root system} is needed (cf.\ Subsection \ref{SSsupp}).\ With these definitions in place we then define a \textit{balanced root system}, independently of any lattice or VOA.\

\medskip
Thus we may summarize Niemeier's observation
by saying that the root system of an even, unimodular lattice of rank $24$ is balanced.\ Similarly, Schellekens' result is that if $V$ is a holomorphic VOA with $c=24$ then either
the root system of $V$ is balanced or else $V_1=0$ and there is no root system.

\medskip
The rest of the present paper is devoted to applying these ideas to holomorphic VOAs with $c=32$ and $40$.\footnote{The central charge of a holomorphic VOA is a positive integer divisible by $8$.}\ Although the list of holomorphic VOAs having $c\leq 24$ is, as described in the first paragraph, quite manageable, the reader should understand that for higher values of $c$  this is far from true.\ See Table \ref{T1} for a few estimates.\ The entry for $c=32$ arises from  mass formulas for the corresponding holomorphic lattice VOAs  (aka even, unimodular lattices of rank $32$)  due to \cites{OK1}.\ The analogous (weak) estimate for $c=40$ is lifted from a remark in \cites{G}.

\medskip
We say that a holomorphic VOA is \textit{balanced} if its associated root system is balanced (cf.\ Subsection \ref{SSbalancedVOA}).\ Because of the importance of the balance property for VOAs with $c=24$ it is natural to ask about this property for higher values of $c$.\ To be blunt, how many balanced holomorphic VOAs are there, say with $c=32$?\ King's results \cites{OK2,OK1} give mass formulas for even unimodular lattices of rank $32$ with a specified root system, so we can give an  answer to this question if we restrict ourselves to  lattice theories.\ The data from \textit{loc.\ cit.}\ delivers a mixed message that we find provocative.\ Most root systems attached to rank $32$ even unimodular lattices are \textit{not} balanced:\ exactly 13218 root systems occur and of these, just 82 are balanced.\ On the other hand, lattices with a balanced root system predominate when counted according to their  mass:\ over $97\%$ of them have an extended root system of type
$O^{32-n}A_1^n$ for some $n$.\ For further discussion see Subsection \ref{SSPP}.\ We refer the reader to \cites{AMM32all,AMM40all} for the complete list of balanced root systems for $c=32$ and $c=40$. 

\medskip
Our main result is a contribution towards the problem of describing the root systems that may occur in a balanced holomorphic VOA of central charge $32$ and as such it is entirely analogous to Schellekens' list \cites{S}.\ We summarize some of our main results as follows:

\begin{thm}\label{thmpurepower}
Suppose that $V=\CC\mathbf{1}\oplus V_1\oplus...$ is a balanced, holomorphic VOA with $c=32$ such that $V_1$ is a semisimple Lie algebra.\ Then $\dim V_1\geq32$, and if $\dim V_1\geq 56$ then the root system is among the following:
\begin{eqnarray*}
&&(a)\  A_{16, 1}^2, A_{8, 1}^4, A_{4, 1}^8, A_{2,1}^{16}, A_{1, 1}^{32},  D_{32, 1}, D_{16, 1}^2, D_{8, 1}^4, D_{4, 1}^8, D_{16, 1}E_{8, 1}^2, E_{8, 1}^4.\\
&&(b)\ A_{1, 4}^{16}, B_{2, 2}^8, B_{4, 2}^4, B_{8, 2}^2.\\
&&(c)\ C_{4, 4}^2, D_{4, 8}^2, F_{4, 4}^2, A_{1,1}^{12}A_{3,2}^4
\end{eqnarray*}
\end{thm}

\medskip
For further discussion of this result, see Section \ref{secCompute}.\ Concerning the range $32\leq\dim V_1\leq 55$, the absence of any conclusions here is due simply to a lack of computational resources.\ 
Apart from a few cases, similar comments also apply when $c=40$, cf.\ \cites{AMM40all}.

\medskip
The root systems in (a) are those occurring in balanced lattice theories and those in (b) occur in balanced DGM orbifolds (see Section \ref{SDGM}).\ Whether or not there are balanced holomorphic VOAs with one of the root systems in (c) remains open, although there is  corroborating numerical evidence favoring existence, namely they passed all of our tests.\ See Section \ref{secCompute} for further details.

\medskip
We briefly discuss how the proof of the Theorem proceeds.\ Generally the (abelian) radical of $V_1$, call it $A$, is nonzero.\ The case
that $V_1=A$ is readily dealt with (cf.\ Theorem \ref{thmL}) so the main issues are concerned with the situation that $A\subsetneq V_1$.\ In this case, by \cite{M}*{Theorem 5} the subVOA $\langle V_1\rangle$ of $V$ generated by $V_1$ is properly contained in a  subVOA 
\begin{eqnarray}\label{tensordecomp1}
U:=V_L  \otimes L_{\hat{\mathfrak{g}}_1}(k_1, 0)\otimes...\otimes L_{\hat{\mathfrak{g}}_n}(k_n, 0)
\end{eqnarray}
where $L$ is an even lattice with no roots,
$(V_L)_1=A$, and the $L_{\hat{\mathfrak{g}}_i}(k_i, 0)$ are WZW models of level $k_i$.\ Indeed, $U$ determines the extended root system (\ref{extRS}) where $f:=\rk L$.\ We refer the reader to Subsection \ref{SSsupp} for further details.\ (\ref{tensordecomp1}) is used in showing that $U$ contains the Virasoro vector of $V$, which is a key ingredient in analyzing the possible root systems that may arise, regardless of the nature of $A$.\ This property of the Virasoro vectors means that we can decompose $V$ into a finite, graded, direct sum of simple $U$-modules.\ 
 Nevertheless, we have thus far not been able to use this to distill the lists of balanced root systems in \cites{AMM32all,AMM40all} to  more realistic sizes except in the case $A=0$ and $c=32$ covered by Theorem \ref{thmpurepower}.

\medskip
Finally, we do not consider \textit{uniqueness} questions in this paper.\ Even for lattice theories uniqueness (up to isomorphism) of a balanced holomorphic VOA with a given root system is far from true \cites{OK1} and the best one may hope for is a mass formula.

\begin{table}
\caption{Number of semisimple balanced root systems}
\label{TB}
\begin{center}
\begin{tabular}{|c|c |} \hline
c & $\Phi$  \\ \hline 
$8$ & $16$    \\
$16$ & $60$  \\
$24$ & $221$ \\
$32$ & $449$ \\
$40$ & $1277$ \\ \hline
\end{tabular}
\end{center}
\end{table}

 \begin{table}
\caption{Number of holomorphic VOAs}
\label{T1}
\begin{center}
\begin{tabular}{|c|c|c|} \hline
$c$ & $V_L$ & $V$ \\ \hline 
$8$ & $1$ & $1$   \\
$16$ & $2$ & $2$ \\
$24$ & $24$& $\geq 71$ \\
$32$ & $\geq 10^9$ & $\geq 10^9$\\
$40$ & $\geq 10^{51}$ & $\geq 10^{51}$\\ \hline
\end{tabular}
\end{center}
\end{table}

\medskip
The remainder of the paper is organized as follows:\ in Section \ref{SBR} we introduce extended and balanced root systems, and in Section \ref{Schar} we recall some basic results on the characters of holomorphic VOAs and related topics.\ The main results concerning Virasoro vectors, which are the main technical tools needed to prove Theorem \ref{thmpurepower}, are Proposition \ref{propab} and Theorem \ref{thmbalab}.\ They are both contained in Subsection \ref{SSVirvecs}.\ Section \ref{SDGM} provides some background on DGM orbifolds which, together with King's work (\textit{loc.\ cit.}) facilitates the identification of a DGM orbifold from its root system.\ In Section \ref{secCompute}, we discuss how we make use of our results on balanced VOAs in explicit computer calculations which complete the proof of Theorem \ref{thmpurepower}.

\section*{Acknowledgments}
The authors thank Terry Gannon for useful discussions, especially for pointing out the reference \cites{Gannon} which played a crucial r\^{o}le for the computational part of this work.\ The necessary computer calculations were mainly carried out on machines at the Chair of Algebra and Representation Theory at RWTH Aachen University which the third-named author was kindly allowed to use.

\bigskip\noindent
Mason was partially supported by a Simons grant
$\# 427007$.

\section{Balanced root systems}\label{SBR} For general background material about root systems see for instance \cite{Bourbaki}*{Chapitre VI: Syst\`emes de racines} as well as \cites{H,Kac}.\ Root systems arise naturally in the theory of Lie algebras and they are also widely used in lattice theory and VOA theory.\ In this Section we set up a suitable framework to cover the applications we have in mind.

\subsection{Root systems and Lie algebras}
A finite-dimensional (complex) simple Lie algebra $\mathfrak{g}$ has associated to it a simple root system $\Phi$ and a finite-dimensional semi-simple Lie algebra 
\begin{eqnarray*}
  \mathfrak{g}=\mathfrak{g}_1\oplus...\oplus\mathfrak{g}_n   
\end{eqnarray*}
(the summands $\mathfrak{g}_i$ will always denote the simple components of $\mathfrak{g}$) is similarly associated with a semisimple root system
$$\Phi=\bigcup_{i=1}^n \Phi_i$$
where $\Phi_i$ is the simple root system attached to $\mathfrak{g}_i$.\ It is common practice to denote these objects by the abbreviated notation
$$\mathfrak{g}_1\mathfrak{g}_2...\mathfrak{g}_n\ \  \mbox{or}\ \ \Phi_1\Phi_2...\Phi_n.$$

\medskip 
We extend this notation to the case of finite-dimensional reductive Lie algebras
\begin{eqnarray}\label{redLA}
A\oplus \mathfrak{g}_1\oplus...\oplus\mathfrak{g}_n
\end{eqnarray}
where $A=\CC^f$  is an abelian Lie algebra of dimension $f$.\ It is denoted by
$$O^f\Phi_1\Phi_2...\Phi_n.$$
This practice is often used in lattice theory
e.g., \cites{CS,OK1} where one typically has an even lattice $L$ and its
root sublattice $R$. $R$ is spanned by root vectors comprising a semisimple root system $\Phi_1...\Phi_n$ (of type ADE) and $f:=\rk L-\rk R$.

\medskip
The needs of certain infinite-dimensional Lie algebras and of vertex operator 
algebras motivate an extension of this notation.\
Abstractly, we may take any positive integers
$k_1, ..., k_n$ and simply introduce the symbol, still called a (reductive) root system:
\begin{eqnarray}\label{redRS}
    O^f\Phi_{1, k_1}\Phi_{2, k_2}...\Phi_{n, k_n}.
\end{eqnarray}
$k_i$ is called the \textit{level} of $\Phi_i$.

\medskip
The root system (\ref{redRS}) is still attached to a Lie algebra but it may no longer be finite-dimensional.\ Indeed, the semi-simple part is associated with the tensor product
\begin{eqnarray}\label{WZWprod}
 L_{\widehat{g}_1}(k_1, 0)\otimes ...\otimes  
 L_{\widehat{g}_n}(k_n, 0)
\end{eqnarray}
where $L_{\widehat{g}_i}(k_i, 0)$ is the affine Lie algebra of level $k_i$ determined by $\mathfrak{g}_i$ \cites{Kac,K,LL}.\ Here we follow the notation of \textit{loc.\ cit.}\ This object is also called a WZW model; it is a vertex operator algebra, as is the tensor product \eqref{WZWprod}.\ Indeed, it is a regular VOA \cites{DLM2}. 

\medskip
To deal with the general case of \eqref{redRS} where $f$ may not be zero there are several ways to proceed, e.g., \cites{M}.\ The first is to associate it with the VOA
\begin{eqnarray*}
 M_f\otimes L_{\widehat{g}_1}(k_1, 0)\otimes ...\otimes  
 L_{\widehat{g}_n}(k_n, 0)
\end{eqnarray*}
where $M_f$ is the Heisenberg VOA of rank $f$
\cites{K,LL}.\ This arises naturally though it will not be useful for us because this tensor product is \textit{not} a rational VOA.\ More useful in practice is a further generalization of 
(\ref{redRS}), denoted by
\begin{eqnarray}\label{redRSL}
    L\Phi_{1, k_1}\Phi_{2, k_2}...\Phi_{n, k_n}.
\end{eqnarray}
Here, the new ingredient is an even lattice $L$ of rank $f$ that contains no roots.\ We refer to (\ref{redRSL}) as a \textit{supplemented root system}.\ Associated to this gadget is the regular VOA
\begin{eqnarray}\label{redRSVL}
    V_L\otimes L_{\widehat{g}_1}(k_1, 0)\otimes ...\otimes  
 L_{\widehat{g}_n}(k_n, 0).
\end{eqnarray}
where $V_L$ is the lattice VOA determined by $L$.\ We discuss this construction in greater detail in Subsection \ref{SSsupp}

\subsection{Balanced root systems} We use the notation established in the previous Subsection.\
We want to carefully introduce the idea of a \textit{balanced} root system.\ We begin with 
a finite dimensional simple Lie algebra 
$\mathfrak{g}$, its root system $\Phi_k$ of level $k$, and the affine algebra
$L_{\widehat{\mathfrak{g}}}(k, 0)$ (regarded as a VOA).\ Associated to this set-up are several numerical invariants that we need.\ These include\\
$\ell:=$ Lie rank of $\mathfrak{g}$,\\  
$h^{\vee}:=$ dual Coxeter number of $\Phi$,\\ $d:=\dim\mathfrak{g}$,\\  
$c:=$ central charge of\ $L_{\widehat{\mathfrak{g}}}(k, 0).$\\
There are relations among these invariants; for example it is well-known \cites{LL} that
\begin{eqnarray}\label{cval}
 c=\frac{dk}{k+h^{\vee}}   
\end{eqnarray}
and we may take this to be the \textit{definition} of the central charge of 
$\Phi_k$.\ For convenience we include  some of the relevant data in Table \ref{TRS}.

\begin{table}
\caption{Root system data}
\label{TRS}
\begin{center}
\begin{tabular}{|c|c|c|} \hline
$\mathfrak{g}$ & $d$ & $h^{\vee}$ \\ \hline 
$A_{\ell}$ & $\ell^2+2\ell$ & $\ell+1$   \\
$B_{\ell}$ & $2\ell^2+\ell$ & $2\ell-1$ \\
$C_{\ell}$ & $2\ell^2+\ell$& $\ell+1$ \\
$D_{\ell}$ & $2\ell^2-\ell$ & $2\ell-2$\\
$E_6$ & $78$ & $12$\\
$E_7$ & $133$& $18$  \\
$E_8$ & $248$ & $30$ \\ 
$G_2$ & $14$ & $4$ \\
$F_4$ & $52$ & $9$ \\ \hline
\end{tabular}
\end{center}
\end{table}

\medskip
Now let us return to a general \textit{nonzero} reductive Lie algebra $A\oplus\mathfrak{g}_{1, k_1}\oplus ...\oplus\mathfrak{g}_{n, k_n}$ and its reductive root system $O^f\Phi_{1, k_1}...\Phi_{n, k_n}$.\ For each index $i$, we let $\ell_i, h_i^{\vee}, d_i, c_i$ be the relevant invariants attached to $\Phi_{i, k_i}$ and $\mathfrak{g}_i$.\ Finally we come to

\begin{dfn}\label{dfnbalRS}
Let $c=f+ \sum_i c_i$ and $d:=f+\sum_i d_i$.\ We say that the root system $O^f\Phi_{1, k_1}...\Phi_{n, k_n}$ is \textit{balanced} if, for every index $i$, the following equality holds:
\begin{eqnarray}\label{balRSdef}
    (c-f)\frac{h_i^{\vee}}{k_i}=d-c.
\end{eqnarray}
If there are \textit{no} simple components then  $c=d=f\neq0$.
\end{dfn}

\medskip
To illustrate this definition we will revisit the classification of Niemeier \cites{N} of the root systems attached to even unimodular lattices of rank $24$.

\medskip
Recall that for an even lattice $L$ and its root sublattice $R$, all components of $R$ are simply-laced and the levels are all taken to be $1$.\ Thus the root system attached to $L$ takes the form (dropping the level from the notation) $0^f\Phi_1...\Phi_n$ and each $\Phi_i$ is of type ADE.\ In the level 1, simply-laced case it follows from (\ref{cval}) and Table \ref{TRS} that we always have $c_i=\ell_i$.\ Niemeier found for unimodular $L$ of rank $24$ that either $c=d=f=24$ (the case of no roots, characterizing the Leech lattice); or else
$f=0$ and $\rk R=\rk L=24$.\ Thus in this case
$c=\sum c_i=\sum \ell_i=\rk R =24$.\ Moreover for every index $i$ Niemeier found that 
\begin{eqnarray}\label{Nob}
   h_i^{\vee} = \frac{d}{24}-1. 
\end{eqnarray}
Thus all even unimodular lattices of rank $24$
have a balanced root system in the sense of Definition \ref{dfnbalRS}.

\medskip
We remark that (\ref{Nob}) was at first an unexplained observation of Niemeier.\ An \textit{a priori} proof was supplied by  \cites{V} using modular forms.

\medskip
At this point we could further illustrate Definition \ref{dfnbalRS} by revisiting the work of \cites{S} where the root systems for strongly regular holomorphic VOAs of central charge $24$ were calculated.\ We will forgo this opportunity here although this story runs along lines parallel to that of Niemeier.\ 

\subsection{Enumerating balanced root systems}\label{SSBRS}
We continue with the notation established in the previous Subsection and introduce a bit more.\
For a rational number $c$ and an integer $f$ satisfying $c\geq f\geq 0, c >0$ set
\begin{eqnarray*}
   && BRS(c, f):=\{\mbox{balanced root systems with specified $c$ and $f$} \}.
\end{eqnarray*}

We give two easy Lemmas about the sets $BRS(c, f)$.

\begin{lem}\label{lemBRSbij}
 There is a natural bijection
 $$BRS(c, f)\stackrel{\sim}{\longrightarrow} BRS(c+1, f+1).$$
\end{lem}
\begin{proof}
This follows from the definitions, the bijection in question being
$$O^f\Phi_{1, k_1}...\Phi_{n, k_n}\mapsto 
O^{f+1}\Phi_{1, k_1}...\Phi_{n, k_n}.$$
\end{proof}

\begin{lem}\label{lemBRSfin}
  Each set $BRS(c, f)$ is \textit{finite}.  
\end{lem}
\begin{proof}
    After Lemma \ref{lemBRSbij} it suffices to prove the result for $f=0$ and all $c$.\
So fix $c$ and  consider an arbitrary semisimple balanced root system
$\Phi_{i, k_1}...\Phi_{n, k_n}$
of central charge $c$.\ For any index $i$ we have 
$$c=\sum_j c_j\geq c_i=\frac{k_id_i}{k_i+h_i^{\vee}}.$$
If $\Phi_i$ is of type $ABCDE$ we see from Table \ref{TRS} that $c_i\geq\ell_i-1$ and therefore for
\textit{any} $\Phi_i$ the Lie rank $\ell_i$ is bounded.\ In particular
there are only finitely many possible choices of $\Phi_1,...,\Phi_n$ and of $d=\sum d_i$.

\medskip
On the other hand, because the root system is balanced and semisimple we have
$$k_i=\frac{h_i^{\vee}}{d/c-1}.$$
By the previous paragraph there are only finitely many choices of $h_i^{\vee}$ and $d$, so the same is true of $k_i$.\ Now
the Lemma follows.
\end{proof}

We refer the reader to Table \ref{tab:numBRS} for the number of balanced root systems for small integral values of $c$ and to
\cites{AMM32all,AMM32semi,AMM40all} for more detailed information when $c=32$ or $40$.

\begin{table}
    \centering
    \begin{tabular}{|c||c|c|c|c|c|c|c|c|c|c|}
    \hline 
        $c$ &  1 & 2 & 3 & 4 & 5 & 6 & 7 & 8 & 9 & 10  \\
        \hline 
        $\# BRS(c,0)$ & 1 & 3 & 3 & 7 & 8 & 13 & 15 & 16 & 20 & 42  \\
        \hline 
        \end{tabular}

        \bigskip
        
        \begin{tabular}{|c||c|c|c|c|c|}
        \hline
        $c$ & 16 & 24 & 32 & 40 & 48 \\
        \hline 
        $\# BRS(c,0)$ & 60 & 221 & 449 & 1277 & 3294 \\
        \hline
    \end{tabular}
    \caption{Number of balanced root systems for small integral $c$ and $f=0$}
    \label{tab:numBRS}
\end{table}

\section{Strongly rational and holomorphic VOAs}\label{Schar}

For general background on VOAs we refer the reader to
\cites{LL,M,K,Z}.\ Here  we present some more specialized results that we will need, especially concerning strongly regular VOAs $V$ of central charge $c$, often with $c=32$ or $40$.\ We carry over \textit{en bloc} the notation
of Sections \ref{SIntro} and \ref{SBR} without comment. 
  
 \medskip
 The precise definition of a balanced VOA is not needed at first and it will be deferred until Subsection \ref{SSbalancedVOA}.
  
\subsection{Modular forms for $\SLZ$}
In this subsection we recall some basic facts about modular forms for the full modular groups $\SLZ$, mainly in order to fix notation.\ For additional background we refer the reader to two of the many textbooks written about modular forms, e.g. \cites{CohenStromberg,Zagier123}.

\medskip
A \emph{modular form} of integer weight $k\in\ZZ$ for $\SLZ$ (or level $1$) is a holomorphic function $f:\mathfrak H\to\CC$ ($\mathfrak H$ denotes the complex upper half-plane) satisfying the transformation law
$$f\left(\frac{a\tau+b}{c\tau+d}\right)=(c\tau+d)^kf(\tau)\quad \text{for all }\begin{pmatrix} a & b \\ c & d\end{pmatrix}\in\SLZ,$$
as well as admitting a Fourier expansion of the form
$$f(\tau)=\sum_{n\geq 0} a_nq^n,\quad q=e^{2\pi i\tau}.$$
The fact that this expansion has no principal part is referred to by saying that $f$ is \emph{holomorphic at infinity}.

\medskip
If the Fourier coefficient $a_0=0$, we call $f$ a \emph{cusp form}.\ A \emph{modular function} is a meromorphic function on $\mathfrak H$ transforming like a modular form of weight $0$ and admitting a Fourier expansion with a finite principal part (i.e., $f$ is \emph{meromorphic at infinity}). 

\medskip
An important class of examples of modular forms is given by the \emph{Eisenstein series} of weight $k$,
\begin{gather}\label{eqEisenstein}
E_k(\tau)=1-\frac{2k}{B_k}\sum_{n=1}^\infty \frac{n^{k-1}q^n}{1-q^n},\quad k\geq 4\text{ even},
\end{gather}
where $B_k$ denotes the $k$-th Bernoulli number.\ As the notation suggests, $E_k$ defines a modular form of weight $k$.\ It is a standard fact of the theory that every modular form for $\SLZ$ can be expressed as a polynomial in the Eisenstein series $E_4$ and $E_6$, e.g. the first non-trivial example of a cusp form, the \emph{discriminant function}
$$\Delta(\tau):=\frac{1}{1728}(E_4(\tau)^3-E_6(\tau)^2)=q-24q^2+252q^3-1472q^4+...,$$
which has weight $12$. Note that the definition of the Eisenstein series in \eqref{eqEisenstein} still makes sense for $k=2$, although this function $E_2$ no longer defines a modular form (in fact there are no modular forms of weight $2$ for $\SLZ$), but rather what is called a \emph{quasimodular form}. Even though it is not itself modular, the Eisenstein series of weight $2$ is still an important and useful function which will make occasional appearances below.

\medskip
In the context of VOAs, it is often useful to consider the following renormalisation of the Eisenstein series (see e.g. \cite{Z}*{p. 262}):

\begin{gather}\label{eqEkhat}
    \widehat E_k(\tau)=-\frac{B_k}{k!}E_k(\tau)=-\frac{B_k}{k!}+\frac{2}{(k-1)!}\sum_{n=1}^\infty \frac{n^{k-1}q^n}{1-q^n}.
\end{gather}

Another important example of a modular form for $\SLZ$ is the Dedekind $\eta$-function,
\begin{gather}
    \label{eqeta}
    \eta(\tau):=q^{1/24}\prod_{n=1}^\infty (1-q^n)=q^{1/24}(1 - q - q^2 + q^5 + q^7 - q^{12} - q^{15} + O(q^{17})),
\end{gather}
which is not strictly speaking a modular form in the sense defined above, but is an example of a modular form of weight $\frac 12$ with \emph{nebentypus}, i.e., upon fixing a branch of the holomorphic square-root, one has the transformation behaviour
$$\eta\left(\frac{a\tau+b}{c\tau+d}\right)=\nu(\gamma)(c\tau+d)^{1/2}\eta(\tau),\quad \gamma=\begin{pmatrix} a & b \\ c & d\end{pmatrix}\in\SLZ,$$
where $\nu(\gamma)$ is a certain $24$th root of unity depending on $\gamma$. Note that this implies that $\eta^{24}$ is a cusp form of weight $12$ in the above sense, wherefore one has the important identity $\eta^{24}=\Delta$.

\subsection{Trace functions and the character of $V$}
 In this subsection we make use of the 'square bracket' grading on $V$.\ For further background on this topic see \cites{DLM,M,MT,Z}.\ For a state
 $v\in V_k$ the zero mode of $v$ is, as usual, defined to be $o(v):=v(k-1)$.\ Then
 $o(v):V_n\rightarrow V_n$ for all $n$ and we set
 $$Z(v)= Z(v, \tau):=q^{-c/24}\sum_{n\geq0} Tr_{V_n} o(v)q^n,$$
 We extend $Z$ to a function on $V$ by linearity.\ We need these ideas only for \textit{holomorphic} VOAs, in which case for $v\in V_{[k]}$, Zhu's Theorem \cites{Z} says that $Z(v)$ is a modular form of level $1$ and weight $k$ (with nebentypus). 
 
 \medskip
 The character of $V$ is defined to be
$Z(\mathbf{1})$.\ We set
$$Z_V = Z_V(\mathbf{1}, \tau) = q^{-c/24}\sum_{n\geq0} \dim V_n q^n.$$

In the following we let
 $$J(\tau):=\frac{E_4(\tau)^3}{\Delta(\tau)}-744=q^{-1}+0+196884q+\hdots$$
 be the modular invariant of level $1$ with constant term equal to $0$.
 
 \begin{lem}\label{lem5} If $V$ is a holomorphic VOA of central charge $c$ then the following hold.\\
 (a) If $c=32$ then
 \begin{eqnarray*}
&&Z_V = \frac{E_4(\tau)}{\eta(\tau)^8}(J(\tau)+\dim V_1-248) \\
&&= q^{-4/3}(1+\dim V_1\cdot q+(248\dim V_1+139504)q^2+ ...)
\end{eqnarray*}
(b) If $c=40$ then
 \begin{eqnarray*}
&&Z_V = \frac{E_4(\tau)^2}{\eta(\tau)^{16}}(J(\tau)+\dim V_1-496)\\
&&= q^{-5/3}(1+\dim V_1\cdot q+(496\dim V_1 + 20620)q^2+...)
\end{eqnarray*}
 \end{lem}
 \begin{proof}
 We have already explained that $Z_V(\tau)=q^{-c/24}(1+\hdots)$  is a modular function of weight $0$ and level $1$ (with nebentypus).\ Therefore if $c=32$ then
 $\eta^{32}Z_V(\tau)=1+\hdots$ is a holomorphic modular form of weight $16$  of level $1$.\ Consequently
 $\eta^{32}Z_V(\tau)=
 E_4^4(\tau)+\alpha' E_4(\tau)\Delta(\tau)$
 for some scalar $\alpha'$ and therefore
 $$Z_V(\tau)= \frac{E_4(\tau)}{\eta^8(\tau)}\left( J(\tau)+\alpha \right) =q^{-1/3}(1+248q+...) (q^{-1}+\alpha+...)$$ 
 for some other scalar $\alpha$.\ Comparing constant terms then shows that $\alpha=\dim V_1-248$ and this completes the proof of part (a). 

 \medskip
 The proof of part (b) is completely parallel upon noting that when $c=40$, $\eta^{40}Z_V=1+\hdots$ is a holomorphic modular form of level $1$ and weight $20$ and therefore equal to $E_4^5+\alpha'E_4^2\Delta$.\ The Lemma follows.
\end{proof}

Note that it is possible to extend the notion of VOA characters to the setting of Jacobi forms, see Section \ref{secjactest}.

\subsection{The supplemented root system}\label{SSsupp}

We have introduced supplemented root systems in (\ref{redRSL}) and (\ref{redRSVL}).\ Here we provide some more details which will be utilized in the next Subsection. Further details may be found in \cites{M}.

\medskip
Recall that a strongly regular, simple VOA $V$ is such that the Lie algebra $V_1=A\oplus S$ is reductive with abelian radical $A$ and Levi factor $S$ as in (\ref{redLA}).\ Moreover the subalgebra $\langle V_1\rangle$ of $V$ generated by $V_1$ satisfies
$$\langle V_1\rangle = M_f\otimes L_{\widehat{\mathfrak{g}}_1}(k_1, 0)\otimes\hdots\otimes L_{\widehat{\mathfrak{g}}_1}(k_1, 0) $$
where $M_f$ is the Heisenberg VOA of rank $f$
and $(M_f)_1=A$ generates $M_f$.

\medskip
 $V$ carries (\cites{Li}) a (unique) invariant bilinear form $\langle\ , \ \rangle$ normalized so that $\langle\mathbf{1}, \mathbf{1}\rangle=-1$.\ It is nondegenerate, and we call it the Li bilinear form.\ The Virasoro vector of $M_f$ is 
$$\omega_A:=1/2\sum_{i=1}^f h^i(-1)^2\mathbf{1}$$
where $\{h^i\}$ is an orthonormal basis of $A$ with respect to $\langle\ , \ \rangle$.

\medskip
Let $\mathcal{P}_A$ be the poset consisting of the subVOAs of $V$ having Virasoro vector $\omega_A$, ordered by inclusion.\ \ By Theorem 1 of \cites{M}
$\mathcal{P}_A$ has a unique maximal element, call it $M$, and $M$ enjoys the following properties:
\begin{enumerate}
\item[(a)] $M = C(C(M_f))$ (double commutant). 
\item[(b)] $M\cong V_L$ for some positive-definite even lattice $L$ with no roots.
\end{enumerate}

Since $A$ and $S$ are orthogonal then $M_f$ and the affine algebra 
$\langle S\rangle =
L_{\hat{\mathfrak{g}}_1}(k_1, 0)\otimes \hdots
\otimes L_{\hat{\mathfrak{g}}_n}(k_n, 0) $ commute.\ In other words  $\langle S\rangle\subseteq C(H)$ and therefore
$M$ and $\langle S\rangle$ commute by property (a) above.\ Thus $V$ has a rational subVOA isomorphic to
$$V_L  \otimes L_{\hat{\mathfrak{g}}_1}(k_1, 0)\otimes \hdots
\otimes L_{\hat{\mathfrak{g}}_n}(k_n, 0).$$
It corresponds to the supplemented root system 
$$L\Phi_{1, k_1}...\Phi_{n, k_n}.$$


\subsection{Some Virasoro vectors}\label{SSVirvecs}
The main result of this Subsection is Proposition \ref{propab}.\
We gather here the data associated with the current setting.\ We retain earlier notation and in particular
$V_1=A\oplus S$ and $S=\bigoplus_i\mathfrak{g}_{i, k_i}$  is the Levi factor of the Lie algebra $V_1$.\ In addition, introduce the following notation:
{\allowdisplaybreaks
\begin{eqnarray*}
&&(\ , \ )_i:=\mbox{invariant bilinear form on}\ \frakg_i\ \mbox{normalized so that a}\\
 && \mbox{long root $\alpha$ satisfies}\ (\alpha, \alpha)_i=2,\\
 &&\kappa_i:=\ \mbox{Killing form for}\ \frakg_i,\\
 &&\kappa:=\mbox{Killing form for}\ V_1,\\
&& \omega_i:= \mbox{Virasoro element of}\  L_{\widehat{\mathfrak{g}_i}}(k_i, 0)=\frac{1}{2(k_i+h_i^{\vee})}\sum_{j=1}^{d_i} u^j(-1)^2\mathbf{1} \\
&&\mbox{where $\{u^j\}$ is an orthonormal basis of $\frakg_i$ with respect to $(\ , \ )_i$},\\
&&\omega_A:=\mbox{Virasoro element of}\ \langle A\rangle =\mbox{Virasoro element of}\ V_L\\
&&=\tfrac{1}{2}\sum_j h^j(-1)^2\mathbf{1}\ \mbox{where $\{h^j\}$ is an orthonormal basis of $A$ with respect to}\ \langle \ , \ \rangle,\\
&&\omega_{aff}:=\mbox{Virasoro element of}\ S=\sum_i\omega_i,\\
&&\omega':=\omega-\omega_{aff},\\
&&\omega_T:=\mbox{Virasoro element of}\ \langle V_1\rangle =\omega_A+\omega_{aff}, \\
&&\omega'':=\omega-\omega_T, \\
&&c_{aff}:=\mbox{central charge of}\ S = \sum_i c_i,\\
&&c_A:=\mbox{central charge of}\ \langle A\rangle=
\dim A= \rk L, \\
&&c_T:= c_A+c_{aff}, \\
&&Y(\omega, z)=\sum_{n\in\mathbf{Z}} L(n)z^{-n-2}, \\
 &&Y(\omega', z)=\sum_{n\in\mathbf{Z}} L'(n)z^{-n-2}, \\
 &&Y(\omega_{aff}, z)=\sum_{n\in\mathbf{Z}} L_{aff}(n)z^{-n-2},\\
 &&Y(\omega'', z):=\sum_{n\in\mathbf{Z}} L''(n)z^{-n-2}, 
\end{eqnarray*}
}

Note that $\omega_T$ is a Virasoro element, but (despite the notation) it is not \textit{a priori} clear that $\omega''$ is also a Virasoro element although it turns out that it is when $c=32$ or $40$.

\medskip
From the commutator formula for $V$ we have for 
$u, v\in V_1$ that
\begin{align*}
   & [u(m), v(n)] = (u(0)v)(m+n)+m\delta_{m, -n}u(1)v,\\
   \mbox{i.e.,}\ \ &[u(m), v(n)] = [u, v](m+n)+m\delta_{m, -n}\langle u, v\rangle\mathbf{1}.
\end{align*}
Whereas  we have for $a, b\in\frakg_i$ that
\begin{align*}
&  [a_m, b_n] =[a, b]_{m+n}+k_i(a, b)_im\delta_{m, -n}
\end{align*}
A comparison then shows that
\begin{align}\label{compare1}
 &k_i.(a, b)_i=\langle a, b\rangle.   
\end{align}

\begin{lem}
 We have
 \begin{align*}
  & L(2)\omega_i=\frac{c_i}{2}\mathbf{1}.   
 \end{align*}
\end{lem}
\begin{proof} Let us compute
\begin{align*}
    & L(2)u^j(-1)^2\mathbf{1} = [L(2), u^j(-1)^2]\mathbf{1} \\
    &= [L(2), u^j(-1)]u^j(-1)\mathbf{1}+u^j(-1)[L(2), u^j(-1)]\mathbf{1}\\
    &= u^j(1)u^j(-1)\mathbf{1}+u^j(-1)u^j(1)\mathbf{1}\\
    &=u^j(1)u^j\\
    &=\langle u^j, u^j\rangle\mathbf{1}\\
    &=k_i\mathbf{1}.
\end{align*}
Here, in the penultimate equality we used the definition of the Li bilinear form $\langle \ , \ \rangle$.\ The last equality follows from (\ref{compare1}) because we have $(u^j, u^j)_i=1$.

\medskip
Now we obtain
\begin{align*}
    &L(2)\omega_i = \frac{1}{2(k_i+h_i^{\vee})}\sum_{j=1}^{d_i}
    L(2)u^j(-1)^2\mathbf{1}\\
    &=\frac{1}{2(k_i+h_i^{\vee})}d_ik_i\mathbf{1}
    =\frac{c_i}{2}\mathbf{1}.
\end{align*}
The Lemma is proved.
\end{proof}

\begin{cor}\label{cor7}
 $L(2)\omega_{aff}= \frac{1}{2}c_{aff}\mathbf{1}$. $\hfill\Box$
\end{cor}

The vertex operators $Y(h^j, z)$ behave as if they have ``level'' $1$.\ That is, the same proof as the previous Lemma applies when $k_i$ is replaced by $1$ and $u^j$ by $h^j$.\ Hence 
\begin{align}\label{S2}
    L(2)\omega_A = \frac{c_A}{2}\mathbf{1}.
    \end{align}
This observation together with Corollary \ref{cor7} implies
\begin{cor}\label{cor8}
    $L(2)\omega_T=\frac{c_T}{2}\mathbf{1}$.
    $\hfill\Box$
\end{cor}

In what follows we use again the 'square bracket' formalism.\ For further details, see
\cites{DLM,M,MT,Z}.\ We now prove
\begin{prop}\label{propab}
Suppose that $V$ is holomorphic and $c=32$ or $40$.\ Then the following are equivalent:\\
(a) $\omega = \omega_T$, (b) $c=c_T$.
\end{prop}
\begin{proof}
(a)$\Rightarrow$ (b) is obvious so we only need to establish that (b)$\Rightarrow$(a).\ We assume (b) for the rest of the proof.

\medskip\noindent
\underline{Step 1}: $L(2)\omega''=0$.\\
Proof:\ Using Corollary \ref{cor8} we have
$$L(2)\omega''=L(2)(\omega-\omega_T)=\tfrac{c}{2}\mathbf{1}-\tfrac{c_T}{2}\mathbf{1}=0.$$

\medskip\noindent
\underline{Step 2}:\ $\omega''$ is a highest weight vector of weight $2$ for both sets of Virasoro operators 
$\{L(n)\}$ and $\{L[n]\}$.\\
Proof:\ We assert that it suffices to prove that 
$L(1)\omega''=0$.\ For if this holds then by Step 1
both $L(1)$ and $L(2)$ annihilate $\omega''$ in which case every $L(n)\ (n\geq 1)$ annihilates $\omega''$.\ It is then a standard consequence (\textit{loc.\ cit.}) that every
$L[n]\ (n\geq 1)$ also annihilates $\omega''$.\ So $\omega''$ is a highest weight vector in both cases
and since $\omega''\in V_2$ then Step 2 is proved. 

\medskip
We always have $L(1)\omega=0$, and since $V$ is strongly regular then $L(1)V_1=0$, in particular $L(1)u^j=0$.\ It then follows that $L(1)\omega_i=0$
so that also $L(1)\omega_{aff}=0$.\ Similarly we have $L(1)\omega_A=0$.\ Therefore
$L(1)\omega''=L(1)(\omega-\omega_A-\omega_{aff})=0$.\ This completes the proof of Step 2.

\medskip\noindent
\underline{Step 3}: $Z(\omega'')=0$.\\
Proof:\ By step 2, $\omega''$ is a state of square bracket weight $2$ and it then follows from \cites{Z} that $Z(\omega'')$ is a level $1$ modular form (with nebentypus) of weight $2$.\ 

\medskip
We claim that this form is holomorphic at infinity.\ Certainly $o(\omega'')$ annihilates the vacuum, so the coefficient of $q^{-c/24}$ vanishes.\ To say
that the coefficient of $q^{1-c/24}$ also vanishes means that $o(\omega'')$ annihilates $V_1$.\ 
To see this, note that $o(\omega'')=\omega''(1)=
L(0)-L_T(0)$.\ But both $L(0)$ and $L_T(0)$ act as $1$ on $V_1$, so that indeed $o(\omega'')$ annihilates $V_1$.

\medskip
It follows that 
$$Z(\omega'')= \alpha q^{2-c/24}+\hdots$$
for some scalar $\alpha$.\ If $c=32$ then
$\eta^8Z(\omega'')=\alpha q+\hdots$ is a cusp form of level $1$ and weight $6$.\ Since the only nonzero level $1$ cusp forms have weight $\geq 12$ this shows that $Z(\omega'')=0$ as needed.\ Similarly if $c=40$ then $\eta^{16}Z(\omega'')$ is a cusp form of weight $10$ and again $Z(\omega'')=0$.\ This completes the proof of Step 3.

\medskip\noindent
\underline{Step 4}: $\omega''$ is a Virasoro vector of central charge $0$.\\
Proof:\  Use the commutator formula for $V$ to see that
\begin{align*}
&[L(m), L_T(n)] = [\omega(m+1), \omega_T(n+1)]\\
&=\sum_{i\geq0} \binom{m+1}{ i}(\omega(i)\omega_T)(m+n+2-i)\\
&=(L(-1)\omega_T)(m+n+2)+2(m+1)\omega_T(m+n+1)\\
&\qquad\qquad\qquad\qquad\qquad\qquad\qquad +\binom{m+1}{3}(L(2)\omega_T)(m+n-1).
\end{align*}
Here we used $L(1)\omega_T=0$, which follows from
Step 2.\ Therefore, using Corollary \ref{cor8},
\begin{align}\label{brackform}
  &[L(m), L_{T}(n)] =  (m-n)\omega_{T}(m+n+1)+
  \binom{m+1}{3}\delta_{m, -n}\frac{c_{T}}{2}Id_V\notag \\
  &=  (m-n)L_{T}(m+n)+
  \frac{m^3-m}{12}\delta_{m, -n}c_{T}Id_V.
\end{align}

We use (\ref{brackform}) in the next calculation:\begin{align*}
&  [L(m)-L_{T}(m), L(n)-L_{T}(n)] \\
&= \left\{(m-n)L(m+n)+\frac{m^3-m}{12}c\delta_{m, -n}Id_V\right\}\\
&\qquad\qquad+\left\{(m-n)L_{T}(m+n)+\frac{m^3-m}{12}c_{T}\delta_{m, -n}Id_V\right\}\\
&-\left\{(m-n)L_{T}(m+n)+\frac{m^3-m}{12}c_{T}\delta_{m, -n}Id_V\right\}\\
&\qquad\qquad+\left\{(n-m)L_{T}(m+n)+\frac{n^3-n}{12}c_{T}\delta_{m, -n}Id_V\right\}\\
&=(m-n)(L(m+n)-L_{T}(m+n))+\frac{m^3-m}{12}\delta_{m, -n}(c-c_{T}) Id_V.
\end{align*}

This says precisely that $\omega''=\omega-\omega_T$
is a Virasoro vector of central charge $c-c_T=0$,
and the proof of Step 4 is complete.

\medskip\noindent
\underline{Step 5}: All eigenvalues of $L''(0)$
are real.\\
Proof: Since $L''(0)=L(0)-L_A(0)-L_{aff}(0)$ it suffices to prove that all eigenvalues of $L_A(0)$ and $L_{aff}(0)$ are real.

\medskip
Concerning $L_{aff}(0)$, the needed statement is proved in \cites{DM1} as Step 5 in Proposition 4.1 of
\textit{loc.\ cit.}\ As for $L_A(0)$, the same method of proof as the previous case may also be applied to $L_A(0)$.\ This is because $L_A(0)$ is the zero mode of $\omega_A$, which is the Virasoro vector of the subVOA $V_L\subseteq V$.\ Now $V_L$ is rational
and its finitely many irreducible modules have real (in fact non-negative rational) conformal weights and this is all that is needed to apply the argument of \textit{loc.\ cit.} Now Step 5 follows.

\medskip\noindent
\underline{Step 6}: $Tr_V L''(0)^2q^{L(0)-c/24}=0$.\\ 
Proof:\ Apply Zhu recursion \cites{MT,Z} to see that
\begin{align*}
 & Z(L''[-2]\omega'')=  Z(\omega''[-1]\omega'')=
 Tr_V L''(0)^2q^{L(0)-c/24}+\sum_{m\geq1} \frac{1}{m}
 \widehat{E}_{m+1}Z(L''[m-1]\omega'')\\
 &=Tr_V L''(0)^2q^{L(0)-c/24}+ \widehat{E}_{2}Z(L''[0]\omega'')+\frac{1}{3}\widehat{E}_{4}Z(L''[2]\omega'').
\end{align*}
 After Step 4
 we have $L''(2)\omega''=0$ and by Step 3 we get
 $Z(L''(0)\omega'')=2Z(\omega'')=0$.\ Therefore,
 \begin{align*}
  & Tr_V L''(0)^2q^{L(0)-c/24}=Z(L''[-2]\omega'')   
 \end{align*}
 is a modular form of level $1$ (with nebentypus) and weight $4$.\
 It is holomorphic at infinity because $L''(0)$ annihilates 
 both $\mathbf{1}$ and $V_1$, as we saw in the proof of Step 3.\ So the $q$-expansion begins $\alpha q^{2-c/24}+...$ for a scalar $\alpha$.\ 
 
 \medskip
 Thus, arguing as before, if $c=32$ then 
 $\eta^8Tr_V L''(0)^2q^{L(0)-c/24}$
 is a cusp form of level $1$ and weight $8$ which forces $Tr_V L''(0)^2q^{L(0)-c/24}=0$ as required.

 \medskip
 Now suppose that $c=40$.\ Then
 $\eta^{16}Tr_V L''(0)^2 q^{L(0)-c/24}=\alpha q+\hdots$
 is a cusp form of level $1$ and weight $12$.\ It is thus a scalar multiple of $\eta^{24}$, so we see that
 $$Tr_V L''(0)^2q^{L(0)-c/24}=\alpha \eta^8.$$
 By Step 5 $Tr_V L''(0)^2 q^{L(0)-c/24}$ has only nonnegative Fourier coefficients.\ But this is only possible if $Tr_V L''(0)q^{L(0)-c/24}=0$, because otherwise from the last display we see that $\alpha\neq 0$ is real and $\alpha\eta^8(\tau)=\alpha q^{1/3}(1 - 8q + 20q^2+...)$ has both positive and negative real Fourier coefficients.\ This completes the proof of Step 6.

\medskip
Finally, the combination of Steps 5 and 6 shows that 
we must have $L''(0)=0$.\ But by Step 4, $L''(0)=2\omega''$. Consequently $\omega''=0$, i.e.,
$\omega=\omega_T$ and part (a) of Proposition \ref{propab} holds.\ This completes the proof of the Proposition.
\end{proof}

\subsection{Balanced VOAs}\label{SSbalancedVOA}
We continue with previous notation and in this Subsection we take $V$ to be a strongly regular simple VOA of central charge $c$.\ From Subsection \ref{SSsupp}, $V$ has a canonical supplemented root system attached to it.\ Indeed, we explained there that
$V$ has a \textit{canonical subVOA} isomorphic to
$$V_L  \otimes L_{\hat{\mathfrak{g}}_1}(k_1, 0)\otimes \hdots
\otimes L_{\hat{\mathfrak{g}}_n}(k_n, 0).$$
and the supplemented root system is 
$$L\Phi_{1, k_1}\hdots \Phi_{n, k_n}.$$
Alternatively, the extended root system is
$$O^f \Phi_{1, k_1}\hdots \Phi_{n, k_n}$$
where $\rk L=f$.\ Recall that this means
the radical of $V_1$ satisfies $A=\CC^f$.
\begin{dfn} We say that $V$ is \textit{balanced}
 if its associated root system lies in
 $B(c, f)$ (notation as in Subsection \ref{SSBRS}).
\end{dfn}

There are several ways to state the balanced condition numerically (cf.\ Definition \ref{dfnbalRS}).\ It means that for each index $i$ we have
$$(c-\dim A)\frac{h_i^{\vee}}{k_i}= \dim V_1-c,$$
or equivalently,
\begin{eqnarray}\label{balance3}
   (c-\dim A)\left(1+\frac{h_i^{\vee}}{k_i}\right) = \dim S,
\end{eqnarray}
where $S$ is the Levi factor of $V_1$.\ In case $V_1=A$ is abelian it means that
$c=\dim A$.

\begin{rmk}
It is clear that given any balanced root system  $O^f \Phi_{1, k_1}\hdots \Phi_{n, k_n}$, there is a strongly regular, simple,  balanced VOA $V$ having the specified root system.\ Indeed, we may take the tensor product
$$V=V_L  \otimes L_{\hat{\mathfrak{g}}_1}(k_1, 0)\otimes \hdots
\otimes L_{\hat{\mathfrak{g}}_n}(k_n, 0)$$
where, as usual, $L$ is an even lattice with no roots.\ More generally we could take
\begin{eqnarray}\label{exttp}
V=V_L  \otimes L_{\hat{\mathfrak{g}}_1}(k_1, 0)\otimes \hdots
\otimes L_{\hat{\mathfrak{g}}_n}(k_n, 0)\otimes W
\end{eqnarray}
where $W$ is any strongly regular, simple VOA having $W_1=0$.
\end{rmk}

\subsubsection{The cases $c=32$ and $40$}\label{SSbal32}

From now on our interests will focus on \textit{holomorphic} VOAs $V$ having a specified balanced root system.\ Then we are obliged to limit ourselves to the cases $c=32, 40$ because of Proposition \ref{propab}.\ This result
 facilitates the development of a Schellekens-type list for balanced, holomorphic theories with $c=32$ or $40$ as follows.
\begin{thm}\label{thmbalab}
Suppose that $V$ is a balanced holomorphic VOA with $c=32$ or $40$ and $V_1\neq0$.\ Then 
$\omega = \omega_T$
\end{thm}
\begin{proof}
    After Proposition \ref{propab} it suffices to prove that $c=c_T$.\ We have 
\begin{align*}
c_T&=c_A+c_{aff}\\
& = \dim A+\sum_i \tfrac{k_id_i}{k_i+h_i^{\vee}}\\
&= \dim A+\sum_i \tfrac{d_i}{1+h_i^{\vee}/k_i}\\
&= \dim A+\tfrac{(c-\dim A)}{\dim S}\sum_i d_i\\
&= c.
\end{align*}
where we used (\ref{balance3}) to get the penultimate equality. The Theorem is proved. 
\end{proof}

\begin{rmk}
  The point of Theorem \ref{thmbalab}  is that it ensures that $V$ contains no subalgebras that look like (\ref{exttp})
  unless $W=\CC$ is trivial.\ Put another way, we know from Subsection \ref{SSsupp} that $V$ contains  a balanced strongly regular subalgebra
  \begin{eqnarray}\label{Udef}
  U=V_L  \otimes L_{\hat{\mathfrak{g}}_1}(k_1, 0)\otimes \hdots
\otimes L_{\hat{\mathfrak{g}}_n}(k_n, 0)
\end{eqnarray}
and Theorem \ref{thmbalab} ensures that
$V$ decomposes as a \textit{finite} graded direct sum of irreducible $U$-modules.
\end{rmk}

\subsubsection{The abelian case}
Here we deal with the case that $V$ is balanced and holomorphic and $V_1=A$ is abelian.\ In this case we do not need Theorem \ref{thmbalab} and there is no restriction on $c$, for we have

\begin{thm}\label{thmL}
Suppose that $V$ is a balanced, holomorphic VOA with $V_1=A\neq0$.\ Then for any $c$, $V$ is isomorphic to a lattice theory $V_L$ where $L$ is an even unimodular lattice of rank $c$ with no roots.
\end{thm}
\begin{proof}
We know that the balanced condition in this case means that $\dim A=c$.\ Because $V$ is holomorphic $c$ is equal to the effective central charge $\tilde{c}$ and therefore $V$ satisfies the conditions needed to apply \cites{DM2}. This implies that $V$ is a lattice theory $V_L$.\ $L$ is unimodular because $V$ is holomorphic and $L$ has no roots because $S=0$.
\end{proof}

 \begin{rmk} Regarding unimodular lattices  with no roots, see \cites{KV} for further commentary.\ If $c=32$ there are just 5 isomorphism classes (\textit{loc.\ cit}).
\end{rmk}

\subsubsection{Holomorphic VOAs of pure power type}\label{SSPP} We begin with

\begin{dfn}\label{defpptype}
 A holomorphic VOA $V$ (or $V_1$) is of
 \textit{pure power type} if $V_1\neq0$ and if the extended root system for $V_1$
 is of the form $O^f\Phi_k^n\ (n\geq0)$ for some simple root system $\Phi$ at level $k$.
\end{dfn}

The motivation for introducing this Definition arises from commentary in \cite{OK1}*{Remark 8} to the effect that over $97\%$ of the mass of unimodular lattices of rank $32$ comes from lattices having a root system of type $0^{32-n}A_{1, 1}^n$ for some $n$.\ Furthermore 
pure power types predominate in the statement of Theorem \ref{thmpurepower}.\ Thus it is worthwhile to focus on this special type of root system, which is almost (but not necessarily) balanced.\ In fact we have

\begin{lem}\label{corabc}
If $V$ is a holomorphic VOA of pure power type and  $c=32$ or $40$ then the following are equivalent:\ (a) $\omega=\omega_T$, (b) $c=c_T$, (c) $V$ is balanced.
\end{lem}
\begin{proof} 
    (a)$\Leftrightarrow$(b) is Proposition \ref{propab} and (c)$\Rightarrow$(a) is Theorem
    \ref{thmbalab}.\ We prove that (b)
    $\Rightarrow$(c).\ Indeed the root system for $V_1$ is $O^f\Phi^n_k$ as above.\ If $\mathfrak{g}$ is the simple Lie algebra with root system $\Phi$ then
$$c= c_T= f+c_{aff}=\dim A+\tfrac{nk\dim\frakg}{k+h^{\vee}}$$
  and therefore
  $$(c-\dim A)(1+\tfrac{h^{\vee}}{k})=n\dim \frakg = \dim S.$$
  Thus (\ref{balance3}) is satisfied and $V$ is balanced.\ This completes the proof of the Lemma.
\end{proof}

\begin{rmk}
    Suppose that $V$ is a holomorphic lattice theory of pure power type with $c=32$ or $40$.\ By the very construction of $V$ we know that (a) and (b) hold, so $V$ is balanced by Lemma \ref{corabc}.\ Since the pure power and balanced properties are closely related one might then wonder if, in the situation of Lemma \ref{corabc}, they are equivalent.\ The answer is "almost, but not quite".\ The root systems for balanced holomorphic lattice theories with $c=32$ may be extracted from \cites{AMM32all} or \cites{OK2}.\ There are $82$ cases and among these just four, namely $D_{16, 1}E_{8,1}^2$, $O^4A_{5,1}^4D_{4,1}^2$, $O^8A_{5,1}^4D_{4,1}$, and $O^{23}A_{5, 1}D_{4, 1}$ are not of pure power type.
\end{rmk}


\section{DGM-orbifolds}\label{SDGM}
Let $L$ be an even unimodular lattice of rank $c$ and let $\theta$ be the inverting involution of $L$, i.e. $\theta(v)=-v\ (v\in L)$.\ Beyond $V_L$ itself we have the DGM orbifold
$$V:= V_L^+\oplus (V_L^T)^+$$
where $(V_L^T)$ is the unique simple $\theta$-twisted module for $V$ (\cites{DLM}).\ As usual, $+$ denotes $\theta$-fixed points.\ $V$ is also a holomorphic VOA of central charge $c$ (\cites{DGM}).

\medskip
Given a  root system $O^f\Phi_{1, k_1}...\Phi_{n, k_n}$
we  want to be able to decide whether it is the root system for a balanced holomorphic DGM orbifold.\ The main results are Theorems \ref{thmDGM1} and \ref{balancedDGM}; they allow us to do this.\ In particular the second of these Theorems says that if $V$ is balanced then $L$ is necessarily complete, i.e., the Lie algebra $(V_L)_1$ is semisimple.\  All of this holds for \textit{any} $c$, and to make it effective one needs the root systems for all complete, unimodular lattices of rank $c$.\ For $c=32$ these were furnished by \cites{Ke} and later also by \cites{OK1,OK2}.\ But for $c = 40$ this information seems to be unavailable.

\medskip
  We begin with the decomposition
  \begin{eqnarray} \label{VL1decomp}
  (V_L)_1 = A\oplus \frakg_{1, 1}\oplus\hdots \oplus\frakg_{n, 1}.
  \end{eqnarray}
  Here, as usual, $A$ is the radical and $\frakg_{i}$ is a simple Lie algebra of type $ADE$ with root system $\Psi_i$.

  \medskip
  The first observation is that by construction, the twisted part $V_L^{T+}$ has conformal weight at least $2$ and hence it does not contribute to $V_1$.\ Second, 
  $\theta$ negates all states in $A$ so it does not contribute to $V_1$ either.\ Since $\theta$ leaves each $\frakg_{i, 1}$ invariant we have
  $$ (V_L)_1 = (L_{\widehat{g}_{1}}(1, 0)^+)_1\oplus\hdots 
  \oplus (L_{\widehat{g}_{n}}(1, 0)^+)_1=\frakg_{1, 1}^+\oplus ...\oplus \frakg_{n, 1}^+.$$

  The involutorial fixed-point Lie algebras $L_{\widehat{g}_{\ell_i}}(1, 0)^+$ are well-known \cite{Kac}, \cite{Helg}.
  
  \medskip
  First, if $\frakg_{i}=A_1$ then
  $(\frakg_{i}^+)_1=\CC$, indeed it is known \cite{DN} that 
  $L_{A_1}(1, 0)^+=V_{\ZZ_8}$, although we will not need this result here.\ In all other cases $L_{\widehat{g}_{i}}(1, 0)^+$ is a semisimple affine algebra of some level $k_i$.
Then with some care one may read-off the needed information that pertains in our case, say from the Tables in  \cite{Helg}.

  \medskip
  Actually, in the situation at hand one has additional information.\ Thus if $\alpha$ ranges over $\Psi_i^+$ (positive roots) then $f_{\alpha}:=e^{\alpha}+e^{-\alpha}$ ranges over a basis for 
  $\frakg_{i, 1}^+$, so that
  $\dim\frakg_{i, 1}^+=|\Psi^+_i|$.\ Moreover if $\alpha$ ranges over $M\subseteq\Phi_i$ maximal subject to the condition $\alpha, \beta\in M\Rightarrow \alpha\pm\beta \notin\Psi_i$
  then the $f_{\alpha}$ span a Cartan subalgebra
  of $\frakg_{i, 1}^+$.\ Thus the Lie rank of $\frakg_{i, 1}^+$ is equal to $|M|$.\ This number is uniquely determined by $\Psi_i$ and is well-known (cf.\ \cite{AK}) to be given as follows:
\begin{eqnarray*}
  A_n:\ [(n+1)/2];\  D_n:\ 2[n/2]; \   E_6:\ 4;\ E_7:\ 7;\  E_8:\ 8.
  \end{eqnarray*}

  This Lie-theoretic data uniquely determines the conjugacy class of $\theta$ as an involution of the affine algebra.
  Finally, the levels may be determined in a standard way, or else one may notice that $L_{\widehat{\frakg}_{i}}(1, 0)$ and 
  $L_{\widehat{\frakg}_{i}}(1, 0)^+$ have the same central charge and then apply (\ref{cval}) in both cases, which involves the level in the case of the fixed-point subalgebra.\ Thus we arrive at

 \begin{thm}\label{thmDGM1}We have the following Table ($m\geq2$):
\begin{center}
    \begin{tabular}{|c||c|c|c|c|c|c|c|c|c|}
    \hline
       $\frakg_i$  & $A_1$ & $A_2$ & $A_{2m}$ & $A_{2m-1}$ & $D_{2m}$ & $D_{2m+1}$ & $E_6$ & $E_7$ & $E_8$ \\
       \hline
        $\frakg_{i, k_i}^+$ & $\CC$ & $A_{1,4}$ & $B_{m,2}$ & $D_{m,2}$ & $D_{m,1}^2$ & $B_{m,1}^2$ & $C_{4,1}$ & $A_{7,1}$ & $D_{8,1}$ \\
         \hline
    \end{tabular}
\end{center}
(with the conventions that $D_2=A_1^2$ and $D_3=A_3$). $\hfill\Box$
\end{thm}

We can now prove
\begin{thm}\label{balancedDGM}
Suppose that $L$ is an even unimodular lattice of any rank $c$ such that the DGM orbifold $V:=V_L^+\oplus (V_L^T)^+$ is balanced.\ Then $L$ is a complete lattice, i.e., the root system of $L$ also has rank $c$.
\end{thm}
\begin{proof}
  We keep earlier notation and set $\ell_i=$ Lie rank of $\frakg_i$, $L_i:=\frakg_{i,k_i}^+$.\ The first main point  is the identity
  \begin{eqnarray}\label{latticeid}
      \ell_i= \frac{\dim L_i}{1+(h_i^{\vee}/k_i)}\ \ (\ell_i\geq 2),
  \end{eqnarray}
  where $h_i^{\vee}$ is the dual Coxeter number of any of the simple components of $L_i$.\ To verify this one can just proceed in a case-by-case manner.

  \medskip
  On the other hand, since $V_1$ is balanced with Levi decomposition $V_1=B\oplus T$ ($B$ is the radical and $T$ the Levi factor) then
  by (\ref{balance3})
  \begin{eqnarray}\label{balanceB}
    \frac{\dim T}{1+(h_i^{\vee}/k_i)}= c-\dim B,
  \end{eqnarray}
  and by (the proof of) Theorem \ref{thmDGM1} we know that $\dim B$ is equal to the number of indices $i$ such that $l_i=1$.\ Thus by (\ref{latticeid}) and (\ref{balanceB}) we have
  \begin{eqnarray*}
    && \frac{\dim T}{1+(h_i^{\vee}/k_i)}=
    \dim A+\sum_{i,\ell_i\geq2}\ell_i\\
    &&=\dim A+\sum_{\ell_i\geq2} \frac{\dim L_i}{1+(h_i^{\vee}/k_i)}.
  \end{eqnarray*}
  Now $T=\oplus_{\ell_i\geq2}L_i$ and $h_i^{\vee}/k_i$ is independent of $i$, so the last displayed equations imply that $A=0$, that is, $(V_L)_1$ is semisimple.\ This completes the proof of the Theorem.
\end{proof}

We discuss some typical applications of the last Theorem.\ Suppose first that $L$ is an even unimodular lattice with $c=24$ not equal to the Leech lattice.\ Then in the above notation $V$ is holomorphic with $c=24$ and $V_1\neq0$, whence $V$ is balanced by \cites{DM1}.\ Therefore $L$ is complete by Theorem \ref{balancedDGM}.\ This was first proved by Niemeier \cites{N}; the proof just given is independent of this work.

\medskip
Turning to the case $c=32$, after Theorem \ref{balancedDGM} we can use Kervaire's classification \cites{Ke} of complete unimodular lattices and their root systems at rank $c=32$ to write down all balanced DGM orbifolds with $c=32$.\ The outcome is the list of root systems in part b) of Theorem \ref{thmpurepower}.

\medskip
We finish with one last illustration.\
Consider the root system $A_{1, 1}^{12}A_{3, 2}^4$ that occurs in part c) of Theorem \ref{thmpurepower}.\ We assert that it is \textit{not} the root system of a DGM orbifold.\ If false, and if it is the root system of
$V_L^+\oplus (V_L^T)^+$ for some $L$ then from Theorem \ref{thmDGM1} we deduce that the root system of $L$ must be
$D_{4, 1}^3A_{5, 1}^4$.\ Thus $L$ is a complete lattice, as expected after Theorem \ref{balancedDGM}.\ However by \cites{Ke} or \cites{OK1} there is no $L$ with such a 
root system.\ This proves our assertion.

\section{Computational aspects}\label{secCompute}
We have so far developed the basic theory of balanced (holomorphic) VOAs, mainly when $c=32$ or $40$, but for a \textit{general} reductive $V_1$.\ From now we will further restrict ourselves to the \textit{semisimple} case where 
$A=0$ and 
$$V_1 = \frakg_{1, k_1}\oplus \hdots\oplus\frakg_{n, k_n}.$$
Thus from Subsection \ref{SSbal32}, the subalgebra $U$ (\ref{Udef}) is generated by $V_1$ and reduces to  a tensor product of affine algebras and $V$ decomposes into a finite direct sum of graded irreducible $U$-modules.\ We know from Table \ref{tab:numBRS} that there are $449$ semisimple, balanced root systems with $c=32$, and the goal of the present Section is to describe how we can check that most of these potential root systems do \textit{not}, in fact, arise from a holomorphic VOA.\ The only thing that has prevented us extending this computation to the $1277$ semisimple, balanced root systems with $c=40$ is a dearth of computational resources.\ Indeed some cases with $c=32$ and $\dim V_1<56$ also remain untreated for the same reason.\
See Appendix A for a precise list.

\medskip
Only $19$ balanced root systems having $c=32$ and $\dim V_1\geq 56$ survive all of the tests we describe below.\ These are listed in Theorem \ref{thmpurepower}.\ Some of these root systems are attached to holomorphic lattice theories and some to holomorphic DGM orbifolds.\ The latter 
 are recognized using the results of Section \ref{SDGM}, especially Theorems
 \ref{thmDGM1} and \ref{balancedDGM}.\ Lattice theories are easily identified from their root systems: the Lie rank is $32$, the root systems are ADE type and all levels are $1$ \cites{DM2}.\ In both cases the results of \cites{Ke} and \cites{OK1} are also crucial.

 \medskip
 The $4$ root systems listed in part (c) of Theorem \ref{thmpurepower} pass all of the tests described below, and in particular the character of $V$ (cf. Section \ref{Schar}) may indeed be written as a finite sum of characters of irreducible $U$-modules.\ It seems likely that balanced holomorphic VOAs with these root systems indeed exist.\ This, however, remains open.\ Moreover we do not address questions concerning \textit{uniqueness}.
\subsection{The dimension test}\label{secdimtest}
We keep the notation introduced earlier in this Section.\ The character
$\ch_V(\tau)$ of $V$ is uniquely determined by $d_1:=\dim V_1$ as in Lemma~\ref{lem5}.\ And we have seen we may decompose $V$ into a finite direct sum of irreducible $U$-modules (with multiplicities) which are graded subspaces of $V$
$$V=L_{\widehat \frakg}(k,0)+\sum_{\lambda\neq0}L_{\widehat \frakg}(k,\lambda),$$
where $\lambda=(\lambda_1,...,\lambda_r)$ is a tuple of highest weights $\lambda_i$ for $\frakg_{i,k_i}$ and 
\begin{gather}\label{eqLglambda}
L_{\widehat\frakg}(k,\lambda)=L_{\widehat \frakg_1}(k_1,\lambda_1)\otimes ...\otimes L_{\widehat\frakg_r}(k_r,\lambda_r).
\end{gather}
Thus by Lemma \ref{lem5} once more it must be possible to write the number 
$$m:=\dim V_2-\dim{L_{\widehat{\frakg}}(k,0)_2}=\dim V_1+139504-\dim\textrm{Sym}^2 \frakg$$ 
as a nonnegative integer linear combination of the dimensions of the irreducible modules of conformal weight $2$ for the affine algebra.\ These dimensions can be computed based solely on the data associated with the root systems.\ Essentially the question becomes whether there is a partition of $\dim V_2$ whose parts are restricted to the (finite) set 
$$D=\{\dim M\: |\: M\in\textrm{Irr}(\widehat{V_1}),\ \textrm{wt} (M)=2\}.$$ 
Such a partition cannot possibly exist if $\gcd(D)\nmid m$ and it is guaranteed to exist if $m/\gcd(D)\in\ZZ$ exceeds the Frobenius number of $\frac{1}{\gcd(D)}D$, easily computable bounds for which are known (see for instance \cites{Frob} and the references therein). In order to check explicitly if a partition exists, we use generating functions.\
Of the $449$ balanced root systems for  $c=32$, this simple test alone rules out almost $1/3$ of all possibilities.
\subsection{The Jacobi forms test}\label{secjactest}
The previous test may be refined by using a Jacobi form variant of the above idea.\ For example  this is used in \cites{EMS} to rigorously derive Schellekens's list.

\medskip
Following \cites{KM} we define the Jacobi form version of the character of $V$ as
$$Z_V(\tau,z):=\mathrm{Tr}_V e^{2\pi i h(0)z}q^{L(0)-c/24},$$
where $h\in V_1$ is such that $h(0)$ has integral eigenvalues. 

\medskip
According to \cite{KM}*{Theorem 2} $Z_V(\tau, z)$ transforms like a Jacobi form of weight $0$ and index $m=\langle h,h\rangle/2$ (which, according to loc. cit., is an integer) with a character, i.e., for
$\gamma=\left(\begin{smallmatrix} a & b \\ c & d \end{smallmatrix}\right)\in\SLZ$ we have
\begin{gather}
    \label{eqchjacobi}
    Z_V\left(\frac{a\tau+b}{c\tau +d},\frac{z}{c\tau+d}\right)=\exp\left(2\pi i\frac{\langle h,h\rangle}{2}\cdot\frac{c}{c\tau+d}\right)\chi(\gamma)Z_V(\tau,z).
\end{gather}
Furthermore it is holomorphic on $\mathbf H\times\CC$. In fact we can be more precise about the character and state that $\eta^{c}(\tau)Z_V(\tau,z)$ is a weak Jacobi form of weight $c/2$ and index $\langle h,h\rangle/2$.

For a $\frakg$-module $M$ and $h\in\mathcal H$ define
$$S_M^j(h)=\sum_{\mu\in\Pi(M)}m_\mu \mu(h)^j,$$
$\Pi(M)$ denoting the set of weights for $M$ and $m_\mu$ denoting the multiplicity of the weight $\mu$. In particular $S_M^0(h)=\dim(M).$

\medskip
Analogously to Theorem 6.1 in \cites{EMS}, we have the following relations for the functions $S_{V_2}^j(h).$
\begin{thm}\label{thmS2j}
    Set $S_{i}^j:=S_{V_i}^j(h)$ and $d_i:=\dim(V_i)$. Then the following identities hold:
    \begin{enumerate}
        \item $c=32$.
    \begin{align*}
        d_2&=248d_1+139504\\
        S_2^2&=-496S_1^2+(60d_1+16440)\langle h,h\rangle \\
        S_2^4&=488S_1^4-504S_1^2\langle h,h\rangle+(36d_1+5328)\langle h,h\rangle^2 \\
        S_2^6&=-256S_1^6+900S_1^4\langle h,h\rangle-540S_1^2\langle h,h\rangle^2\\
        &\qquad\qquad+(30d_1+2640)\langle h,h\rangle^3\\
        4S_2^{10}-15S_2^8\langle h,h\rangle&=-64S_1^{10}-120S_1^8\langle h,h\rangle+5040S_1^6\langle h,h\rangle^2-12600S_1^4\langle h,h\rangle^3\\
        &\qquad\qquad +6300S_1^2\langle h,h\rangle^4-(315d_1+20160)\langle h,h\rangle^5
    \end{align*}
    \item $c=40$.
    \begin{align*}
        d_2&=496d_1+20620\\
        S_2^2&=-248S_1^2+(60d_1+1560)\langle h,h\rangle \\
        4S_2^6-5S_2^4\langle h,h\rangle &=-32S_1^6-80S_1^4\langle h,h\rangle+360S_1^2\langle h,h\rangle^2-(60d_1+1200)\langle h,h\rangle^3.
    \end{align*}
    \end{enumerate}
\end{thm}
\begin{proof}
    The function
    $$P(\tau,z):=\exp\left(-(2\pi iz)^2\frac{\langle h,h\rangle}{24}E_2(\tau)\right)\eta^{c}(\tau)\mathrm{ch}_V(\tau,z)$$
    transforms like
    \begin{gather}\label{eqPtrans}
        P\left(\frac{a\tau+b}{c\tau +d},\frac{z}{c\tau+d}\right)=(c\tau+d)^{c/2} P(\tau,z),
    \end{gather}
    which follows immediately from the well-known transformation property
    $$E_2\left(\frac{a\tau+b}{c\tau+d}\right)=(c\tau+d)^2E_2(\tau)+\frac{12}{2\pi i}c(c\tau+d)$$
    of the weight $2$ Eisenstein series. 
    Therefore, $P(\tau,z)$ is a Jacobi-like form of weight $c/2$ and index $0$ (for a definition and background on this see \cites{ZagierJacobi}), so its $n$th Taylor coefficient in $z$ is a modular form of weight $c/2+n$ for $\SLZ$, which easy to see from \eqref{eqPtrans}. The fact that they are holomorphic at $\infty$ follows from the Fourier expansion.\ All the stated identities follow from identifying these modular forms explicitly from their $q^0$ and $q^1$ term, which is possible for the relevant weights $16,18,20,22,26$.

    Note that the identities for $d_2$ in both cases $c=32$ and $c=40$ are already stated in Lemma~\ref{lem5}. 
\end{proof}

\begin{rmk}
    Also note that similar relations can of course be obtained $d_i,S_i^2,S_i^4,S_i^6,$ and $S_i^{8}$ with $S_i^{10}$ for arbitrary $i$ by analyzing the coefficient of $q^i$ in the modular forms mentioned in the proof. However, it is not possible to go beyond the value $j=10$ for $S_i^j$ in the same way because then the relevant spaces of modular forms have dimensions greater than $2$. 
\end{rmk}

Again by decomposing $V_2$ into irreducible $\widehat{\frakg}$-modules,
$$V_2=(L_{\widehat \frakg}(k,0))_2\oplus \bigoplus_\lambda m_\lambda (L_{\widehat\frakg}(k,\lambda))_2,$$
the identities in Theorem~\ref{thmS2j} yield a system of linear equations for the multiplicities $m_\lambda$, say $A\vec m=b$. Since these multiplicities must be non-negative integers, we can rule out all root systems where this system has no solution in non-negative integers. 

\medskip
We use the built-in functionality for solving such (integer) linear programs in MAGMA \cites{Magma}, which in turn uses the lp\textunderscore solve-library \cites{lpsolve} to decide this. In fact there are three stages to the computation:
\begin{itemize}
    \item Check for integral solutions (no restrictions on the sign): This doesn't rely on Linear Program and is carried out using MAGMA's internal functionalities.
    \item Check for positive (not necessarily integral) solutions: For each highest weight $\lambda$, we solve the (real-valued) optimization problem
    $\max m_\lambda$ with the constraint $A\vec m=b$. This either yields an additional upper bound (next to the trivial bound $m_\lambda\leq (\dim V_2-\dim (L_{k,0})_2)/\dim (L_{k,\lambda})_2$) for $m_\lambda$ if there is a non-negative maximum, or if the maximum is negative, there can be no non-negative integer solution to the system to begin with.
    \item Using the improved bounds from the previous step (which are frequently quite restrictive, forcing multiplicities to be $0$ in many cases), we now look for non-negative integer solutions of the system $A\vec m=b$. If there is none, we can rule out the root system. Note however that this computation is often quite time consuming.
\end{itemize}

\subsection{The character test}\label{secchartest}
For the previous two tests, we have only relied on the decomposition of the $V_2$-space as a module over the Lie algebra $\frakg=V_1$. Of course one may also take into account integrable modules of higher weights:\ as we have already explained, it is possible to decompose the full VOA $V$ as a finite direct sum of the finitely many irreducible $U$-modules,
$$V=L_{\widehat \frakg}(k,0)+\sum_{\lambda}L_{\widehat \frakg}(k,\lambda),$$
and we let $m_\lambda$ again denote the multiplicity of the module $L_{\widehat \frakg}(k,\lambda)$ which is described as in \eqref{eqLglambda}. Therefore it is necessary that there is a non-negative integer linear combination of the characters of the modules $L_{\widehat \frakg}(k,\lambda)$ which equals the character of $V$:
$$Z_V(\tau)=Z_{L_{\widehat \frakg}(k,0)}(\tau)+\sum_\lambda m_\lambda Z_{L_{\widehat \frakg}(k,\lambda)}(\tau).$$
In \cite{Gannon}*{Proposition 1} there is an explicit formula for the characters $Z_{L_{\widehat \frakg}(k,\lambda)}$ for a module over a simple Lie algebra $\frakg$,
\begin{equation}
    Z_{L_{\widehat \frakg}(k,\lambda)}(\tau,z)=q^{-c/24}\left(\prod_{n=1}^\infty\sum_{\ell=0}^\infty q^{n\ell}\dim(\operatorname{Sym}^\ell(\frakg))\right){\sum_{\gamma\in \kappa Q^\vee+\lambda}
q^{\langle\gamma,\gamma+2\rho\rangle/(2\kappa)}\ch_{\gamma}(0)}.
\end{equation}
Here, $z\in\mathcal H$ is an element of the Cartan subalgebra of $\frakg$, $\operatorname{Sym}^\ell(\frakg)$ denotes the $\ell$th symmetric power of $\frakg$, $\kappa=k+h^\vee$, $Q^\vee$ is the dual of the coroot lattice of $\frakg$ (so the usual root lattice of $\frakg$ if $\frakg$ is of $ADE$-type; for the remaining cases see \cite{Kac}*{Section 6.7}), $\rho$ is the Weyl vector and $\ch_\gamma(0)$ is obtained by formally plugging in $\gamma$ into the Weyl character formula (see e.g. \cite{Hall}*{Theorem 10.4}).

Using this formula, we can again obtain a linear system of equations for the multiplicities $m_\lambda$, which we can analyse in much the same way as described above. Note that this test has certain advantages in that it does not distinguish between modules which only differ in the order of their tensor factors, so the number of variables can be reduced compared to the Jacobi forms test, sometimes significantly so. By increasing the number of coefficients of the character used in the computation it is possible to adjust the performance of the test to some extent.

\subsection{Summary of the test results}
We implemented the tests described above in \textsc{Magma} and ran them on all 449 semisimple, balanced root systems for $c=32$.\ A total of $121$ root systems were ruled out by the dimension test (see Section \ref{secdimtest}), $245$ by the Jacobi forms test (Section \ref{secjactest}), and $25$ by the character test (Section \ref{secchartest}). The $19$ root systems listed in Theorem~\ref{thmpurepower} passed all three tests and 16 of these may be realized as root systems of holomorphic lattice theories or DGM orbifolds.\ This leaves $38$ root systems unaccounted for. For these, the computational resources the the authors' disposal were insufficient to finish the tests, either due to extremely long computation times (in some cases testing one root system was aborted after several months) and/or memory (exceeding 200 GB RAM).



\newpage
\appendix
\section{Balanced root systems for central charge $32$}
In this Appendix we record the computational results concerning balanced holomorphic VOAs of central charge $32$ described above, in the case that $V_1$ is a semisimple Lie algebra, grouped together by $\dim V_1=d_1$.

There are a total of $\mathbf{449}$ balanced root systems in this case. Of those 449, we were able to rule out $392$. We list all of these explicitly and indicate, whether they can occur as the root system of a holomorphic VOA: The letter $X$ next to a root system means that the root system was ruled out. An additional note about the computational test that ruled out the root system -- ``dim'' for the dimension test, ``Jac'' for the Jacobi forms test, and ``char'' for the character test (cf.\  Section 5) is also given. A checkmark \Checkmark next to a root system indicates that it passed \emph{all} computational tests we tried. Whenever it is known for sure that a holomorphic VOA with that root system exists (because there is a lattice theory or a DGM orbifold realizing the root system), an indication (``lat'' or ``DGM'') is given. A question mark ? next to a root system (38 cases) indicates that there was no conclusive result of the tests due to computational limitations: Except for the root systems $A_{1,64}^{3}A_{4,160}$ and $A_{5,64}$, they all passed the dimension test, though, whereas for the two exceptions, even this relatively simple test could not be carried out using the computational ressources available to the authors.

\newpage
\begin{multicols}{2}
\begin{center}
    \begin{tabular}{|p{0.5cm}p{1.5cm}p{3cm}p{1.5cm}|}
    \hline
    \# & $\dim V_1$ & Root system & Result\\
    \hline
    \hline 
     1 & 33 & $A_{1,64}A_{3,128}^2$ & ?\\
     2 & 33 & $A_{1,64}^{3}A_{4,160}$ & ?\\
     3 & 33 & $A_{2,96}A_{3,128}B_{2,96}$ & X (dim)\\
     4 & 33 & $A_{1,64}^{4}B_{3,160}$ & ?\\
     5 & 33 & $A_{1,64}B_{2,96}^{3}$ & ?\\
     6 & 33 & $A_{1,64}^{4}C_{3,128}$ & ?\\
     7 & 33 & $A_{1,64}A_{2,96}^2G_{2,128}$ & ?\\
     8 & 33 & $A_{1,64}^{3}B_{2,96}G_{2,128}$ & ? \\
     9 & 33 & $A_{1,64}^{3}A_{2,96}^{3}$ & ?\\
    10 & 33 & $A_{1,64}^{6}A_{3,128}$ & ? \\
    11 & 33 & $A_{1,64}^{5}A_{2,96}B_{2,96}$ & ?\\
    12 & 33 & $A_{1,64}^{11}$ & ?\\
\hline 
    13 & 34 &  $A_{4,80}B_{2,48} $ & X (dim)\\
    14 & 34 &  $A_{1,32}B_{2,48}B_{3,80} $  & X (Jac) \\
    15 & 34 &  $A_{1,32}B_{2,48}C_{3,64} $  & X (dim)\\
    16 & 34 &  $A_{1,32}^{2}D_{4,96} $  & X (dim) \\
    17 & 34 &  $B_{2,48}^{2}G_{2,64} $  & X (dim)\\
    18 & 34 &  $A_{1,32}^{2}G_{2,64}^{2} $  & ? \\
    19 & 34 &  $A_{1,32}A_{2,48}^2A_{3,64} $  & X (dim) \\
    20 & 34 &  $A_{2,48}^{3}B_{2,48} $ & X (dim) \\
    21 & 34 &  $A_{1,32}^{3}A_{3,64}B_{2,48} $  & X (dim) \\
    22 & 34 &  $A_{1,32}^{2}A_{2,48}B_{2,48}^{2} $  & X (dim) \\
    23 & 34 &  $A_{1,32}^{4}A_{2,48}G_{2,64} $  & ? \\
    24 & 34 &  $A_{1,32}^{6}A_{2,48}^{2} $  & ? \\
    25 & 34 &  $A_{1,32}^{8}B_{2,48} $  & ? \\
\hline
    26 & 35 & $A_{5,64}$ & ?\\
\hline 
    27 & 36 & $B_{4,56}$ & X (dim) \\
    28 & 36 & $C_{4,40}$  & X (dim)\\
    29 & 36 & $A_{3,32}B_{3,40}$  & X (dim)\\
    30 & 36 & $A_{3,32}C_{3,32}$  & X (dim)\\
    31 & 36 & $A_{2,24}D_{4,48}$ & X (dim)\\
    32 & 36 & $A_{2,24}G_{2,32}^{2}$  & X (dim)\\
    33 & 36 & $A_{1,16}^{2}A_{3,32}^{2}$ & X (Jac) \\
    34 & 36 & $A_{1,16}^{4}A_{4,40}$ & X (Jac) \\
    35 & 36 & $A_{1,16}A_{2,24}A_{3,32}B_{2,24}$ & X (Jac) \\
    36 & 36 & $A_{1,16}^{5}B_{3,40}$ & X (Jac) \\
    37 & 36 & $A_{2,24}^{2}B_{2,24}^{2}$ & X (Jac) \\
    38 & 36 & $A_{1,16}^{2}B_{2,24}^{3}$ & X (Jac) \\
    39 & 36 & $A_{1,16}^{5}C_{3,32}$ & ? \\
    40 & 36 & $A_{1,16}^{2}A_{2,24}^{2}G_{2,32}$ & X (Jac) \\
    \hline
    \end{tabular}
\end{center}

    \begin{center}
        \begin{tabular}{|p{0.5cm}p{1.5cm}p{3cm}p{1.5cm}|}
    \hline
    \# & $\dim V_1$ & Root system & Result\\
    \hline
    \hline 
    41 & 36 & $A_{1,16}^{4}B_{2,24}G_{2,32}$ & X (Jac) \\
    42 & 36 & $A_{1,16}^{4}A_{2,24}^{3}$ & ? \\
    43 & 36 & $A_{1,16}^{7}A_{3,32}$ & ? \\
    44 & 36 & $A_{1,16}^{6}A_{2,24}B_{2,24}$ & ? \\
    45 & 36 & $A_{1,16}^{12}$ & ? \\
\hline
    46 & 38 & $B_{2,16}D_{4,32}$ & X (Jac) \\
    47 & 38 & $A_{2,16}B_{2,16}^{3}$ & X (Jac) \\
\hline
    48 & 40 &  $A_{2,12}^{2}A_{4,20} $ & X (Jac) \\
    49 & 40 &  $A_{3,16}^{2}B_{2,12} $  & X (Jac) \\
    50 & 40 &  $A_{1,8}^{2}A_{4,20}B_{2,12}$ & X (Jac)\\
    51 & 40 &  $A_{1,8}A_{2,12}^2B_{3,20} $  & X (Jac) \\
    52 & 40 &  $A_{1,8}^{3}B_{2,12}B_{3,20} $  & ? \\
    53 & 40 &  $B_{2,12}^{4} $  & X (Jac) \\
    54 & 40 &  $A_{1,8}A_{2,12}^2C_{3,16} $  & X (Jac) \\
    55 & 40 &  $A_{1,8}^{3}B_{2,12}C_{3,16} $ & X (Jac) \\
    56 & 40 &  $A_{1,8}^{4}D_{4,24} $  & X (Jac) \\
    57 & 40 &  $A_{1,8}A_{2,12}A_{3,16}G_{2,16} $ & X (Jac) \\ 
    58 & 40 &  $A_{2,12}^{2}B_{2,12}G_{2,16} $  & X (Jac) \\
    59 & 40 &  $A_{1,8}^{2}B_{2,12}^{2}G_{2,16} $ & X (Jac) \\
    60 & 40 &  $A_{1,8}^{4}G_{2,16}^{2} $  & X (Jac) \\
    61 & 40 &  $A_{2,12}^{5} $ & ? \\
    62 & 40 &  $A_{1,8}^{3}A_{2,12}^{2}A_{3,16} $ & ? \\
    63 & 40 &  $A_{1,8}^{2}A_{2,12}^{3}B_{2,12} $  & X (Jac) \\
    64 & 40 &  $A_{1,8}^{5}A_{3,16}B_{2,12} $ & ? \\
    65 & 40 &  $A_{1,8}^{4}A_{2,12}B_{2,12}^{2} $ & ? \\
    66 & 40 &  $A_{1,8}^{6}A_{2,12}G_{2,16} $ & ? \\
    67 & 40 &  $A_{1,8}^{8}A_{2,12}^{2} $ & ? \\
    68 & 40 &  $A_{1,8}^{10}B_{2,12} $ & ? \\
\hline
    69 & 42 &  $B_{3,16}^{2} $ & X (dim)\\
\hline
    70 & 44 &  $A_{2,8}^{2}D_{4,16} $ & X (dim)\\
    71 & 44 &  $A_{2,8}^{3}B_{2,8}^{2} $  & X (dim)\\
\hline
    72 & 48 &  $A_{6,14} $ & X (dim)\\
    73 & 48 &  $A_{1,4}D_{5,16}  $ & X (dim)\\
    74 & 48 &  $A_{4,10}^{2}$ & X (Jac) \\
    75 & 48 &  $A_{1,4}A_{5,12}B_{2,6} $ & X (Jac) \\
    76 & 48 &  $A_{1,4}A_{4,10}B_{3,10}$ & X (Jac) \\
    77 & 48 &  $A_{1,4}^{4}B_{4,14}$ & X (Jac) \\
    78 & 48 &  $A_{1,4}^{2}B_{3,10}^{2}$ & X (Jac) \\
    79 & 48 &  $A_{1,4}A_{4,10}C_{3,8} $ & X (Jac) \\
    80 & 48 &  $A_{1,4}^{2}B_{3,10}C_{3,8} $ & X (Jac) \\
    \hline 
        \end{tabular}
    \end{center}
\end{multicols}

\begin{multicols}{2}
\begin{center}
    \begin{tabular}{|p{0.7cm}p{1.3cm}p{3cm}p{1.5cm}|}
    \hline
    \# & $\dim V_1$ & Root system & Result\\
    \hline
    \hline 
    81 & 48 &  $A_{1,4}^{4}C_{4,10} $ & X (Jac) \\
    82 & 48 &  $A_{1,4}^{2}C_{3,8}^{2} $ & X (Jac) \\
    83 & 48 &  $B_{2,6}^{2}D_{4,12} $ & X (Jac) \\
    84 & 48 &  $A_{4,10}B_{2,6}G_{2,8} $ & X (dim)\\
    85 & 48 &  $A_{1,4}B_{2,6}B_{3,10}G_{2,8} $ & X (Jac) \\
    86 & 48 &  $A_{1,4}B_{2,6}C_{3,8}G_{2,8} $ & X (Jac) \\
    87 & 48 &  $A_{1,4}^{2}D_{4,12}G_{2,8} $ & X (Jac) \\
    88 & 48 &  $B_{2,6}^{2}G_{2,8}^{2} $ & X (Jac) \\
    89 & 48 &  $A_{1,4}^{2}G_{2,8}^{3} $ & X (char) \\
    90 & 48 &  $A_{1,4}A_{3,8}^3 $ & ? \\
    91 & 48 &  $A_{2,6}^{3}A_{4,10} $ & X (Jac)  \\
    92 & 48 &  $A_{1,4}^{3}A_{3,8}A_{4,10} $ & X (Jac) \\
    93 & 48 &  $A_{2,6}A_{3,8}^2B_{2,6} $ & X (Jac) \\
    94 & 48 &  $A_{1,4}^{2}A_{2,6}A_{4,10}B_{2,6} $ & X (Jac) \\
    95 & 48 &  $A_{1,4}A_{2,6}^3B_{3,10} $ & X (Jac) \\
    96 & 48 &  $A_{1,4}^{4}A_{3,8}B_{3,10} $ & X (Jac) \\
    97 & 48 &  $A_{1,4}^{3}A_{2,6}B_{2,6}B_{3,10} $ & X (Jac)\\
    98 & 48 &  $A_{1,4}A_{3,8}B_{2,6}^{3} $ & X (Jac) \\
    99 & 48 &  $A_{2,6}B_{2,6}^{4} $ & X (Jac) \\
    100 & 48 & $A_{1,4}A_{2,6}^3C_{3,8}$ & X (Jac) \\
    101 & 48 & $A_{1,4}^{4}A_{3,8}C_{3,8}$ & ? \\
    102 & 48 & $A_{1,4}^{3}A_{2,6}B_{2,6}C_{3,8}$ & X (Jac) \\
    103 & 48 & $A_{1,4}^{4}A_{2,6}D_{4,12}$ & X (Jac) \\
    104 & 48 & $A_{1,4}A_{2,6}^2A_{3,8}G_{2,8}$ & X (Jac) \\
    105 & 48 & $A_{2,6}^{3}B_{2,6}G_{2,8}$ & ? \\
    106 & 48 & $A_{1,4}^{3}A_{3,8}B_{2,6}G_{2,8}$ & X (Jac) \\
    107 & 48 & $A_{1,4}^{2}A_{2,6}B_{2,6}^{2}G_{2,8}$ & X (Jac) \\
    108 & 48 & $A_{1,4}^{4}A_{2,6}G_{2,8}^{2}$ & X (char)\\
    109 & 48 & $A_{2,6}^{6}$ & X (char)\\
    110 & 48 & $A_{1,4}^{3}A_{2,6}^{3}A_{3,8}$ & X (Jac) \\
    111 & 48 & $A_{1,4}^{6}A_{3,8}^{2}$ & ? \\
    112 & 48 & $A_{1,4}^{8}A_{4,10}$ & X (char) \\
    113 & 48 & $A_{1,4}^{2}A_{2,6}^{4}B_{2,6}$ & ? \\
    114 & 48 & $A_{1,4}^{5}A_{2,6}A_{3,8}B_{2,6}$ & X (Jac) \\
    115 & 48 & $A_{1,4}^{9}B_{3,10}$ & X (char)\\
    116 & 48 & $A_{1,4}^{4}A_{2,6}^{2}B_{2,6}^{2}$ & X (Jac) \\
    117 & 48 & $A_{1,4}^{6}B_{2,6}^{3}$ & ? \\
    118 & 48 & $A_{1,4}^{9}C_{3,8}$ & ? \\
    119 & 48 & $A_{1,4}^{6}A_{2,6}^{2}G_{2,8}$ & X (Jac) \\
    120 & 48 & $A_{1,4}^{8}B_{2,6}G_{2,8}$ & X (char)\\
\hline 
\end{tabular}
\end{center}

\begin{center}
 \begin{tabular}{|p{0.7cm}p{1.3cm}p{3cm}p{1.5cm}|}
    \hline
    \# & $\dim V_1$ & Root system & Result\\
    \hline
    \hline 
    121 & 48 & $A_{1,4}^{8}A_{2,6}^{3}$ & ? \\
    122 & 48 & $A_{1,4}^{11}A_{3,8}$ & X (char)\\
    123 & 48 & $A_{1,4}^{10}A_{2,6}B_{2,6}$ & ? \\
    124 & 48 & $A_{1,4}^{16}$  & \Checkmark (DGM)\\
\hline
    125 & 56 & $D_{4,8}^{2}$ & \Checkmark \\
    126 & 56 & $A_{2,4}B_{2,4}^{2}D_{4,8}$ & X (Jac) \\
    127 & 56 & $A_{2,4}^{2}B_{2,4}^{4}$  & X (Jac) \\
    128 & 56 & $A_{2,4}^{7}$ & X (char)\\
\hline
    129 & 64 & $A_{1,2}^{3}B_{5,9}$  & X (dim) \\
    130 & 64 & $A_{1,2}^{3}C_{5,6}$ & X (dim)\\
    131 & 64 & $B_{4,7}D_{4,6}$ & X (dim)\\
    132 & 64 & $C_{4,5}D_{4,6}$ & X (dim)\\
    133 & 64 & $A_{1,2}^{4}F_{4,9}$ & X (dim)\\
    134 & 64 & $B_{4,7}G_{2,4}^{2}$ & X (dim)\\
    135 & 64 & $C_{4,5}G_{2,4}^{2}$ & X (dim)\\
    136 & 64 & $A_{2,3}^{2}A_{6,7}$ & X (dim)\\
    137 & 64 & $A_{1,2}^{2}A_{6,7}B_{2,3}$ & X (dim)\\
    138 & 64 & $A_{2,3}A_{5,6}B_{3,5}$ & X (dim)\\
    139 & 64 & $A_{1,2}A_{3,4}B_{2,3}B_{4,7}$ & X (dim)\\
    140 & 64 & $A_{2,3}B_{2,3}^{2}B_{4,7}$ & X (dim)\\
    141 & 64 & $A_{2,3}A_{5,6}C_{3,4}$ & X (dim)\\
    142 & 64 & $A_{1,2}A_{3,4}B_{2,3}C_{4,5}$ & X (Jac) \\
    143 & 64 & $A_{2,3}B_{2,3}^{2}C_{4,5}$ & X (Jac) \\
    144 & 64 & $A_{3,4}B_{3,5}D_{4,6}$ & X (dim)\\
    145 & 64 & $A_{3,4}C_{3,4}D_{4,6}$ & X (dim)\\
    146 & 64 & $A_{1,2}A_{2,3}^2D_{5,8}$ & X (dim)\\
    147 & 64 & $A_{1,2}^{3}B_{2,3}D_{5,8}$ & X (dim)\\
    148 & 64 & $A_{2,3}D_{4,6}^{2}$ & X (dim)\\
    149 & 64 & $A_{3,4}A_{5,6}G_{2,4}$ & X (dim)\\
    150 & 64 & $A_{1,2}^{2}A_{2,3}B_{4,7}G_{2,4}$ & X (Jac) \\
    151 & 64 & $A_{2,3}B_{3,5}^{2}G_{2,4}$ & X (dim)\\
    152 & 64 & $A_{2,3}B_{3,5}C_{3,4}G_{2,4}$ & X (dim)  \\
    153 & 64 & $A_{1,2}^{2}A_{2,3}C_{4,5}G_{2,4}$ & X (Jac) \\
    154 & 64 & $A_{2,3}C_{3,4}^{2}G_{2,4}$ & X (dim)\\
    155 & 64 & $A_{3,4}B_{3,5}G_{2,4}^{2}$ & X (dim)\\
    156 & 64 & $A_{3,4}C_{3,4}G_{2,4}^{2}$ & X (dim)\\
    157 & 64 & $A_{2,3}D_{4,6}G_{2,4}^{2}$ & X (dim)\\
    158 & 64 & $A_{2,3}G_{2,4}^{4}$ & X (dim)\\
    159 & 64 & $A_{2,3}^{2}A_{4,5}^{2}$ & X (Jac) \\
    160 & 64 & $A_{1,2}^{2}A_{2,3}A_{3,4}A_{5,6}$ & X (dim)\\
    \hline
\end{tabular}
\end{center}
\end{multicols}

\newpage

\begin{multicols}{2}
\begin{center}
    \begin{tabular}{|p{0.7cm}p{1.3cm}p{3cm}p{1.5cm}|}
    \hline
    \# & $\dim V_1$ & Root system & Result\\
    \hline
    \hline    
    161 & 64 & $A_{3,4}^{2}A_{4,5}B_{2,3}$ & X (Jac) \\
    162 & 64 & $A_{1,2}^{2}A_{4,5}^{2}B_{2,3}$ & X (Jac) \\
    163 & 64 & $A_{1,2}A_{2,3}^2A_{5,6}B_{2,3}$ & X (dim)\\
    164 & 64 & $A_{1,2}A_{2,3}^2A_{4,5}B_{3,5}$ & X (Jac) \\
    165 & 64 & $A_{1,2}^{3}A_{5,6}B_{2,3}^{2}$ & X (dim)\\
    166 & 64 & $A_{1,2}^{4}A_{2,3}^{2}B_{4,7}$ & X (Jac) \\
    167 & 64 & $A_{1,2}A_{3,4}^2B_{2,3}B_{3,5}$ & X (Jac) \\
    168 & 64 & $A_{1,2}^{3}A_{4,5}B_{2,3}B_{3,5}$ & X (Jac) \\
    169 & 64 & $A_{1,2}^{2}A_{2,3}^{2}B_{3,5}^{2}$ & X (Jac) \\
    170 & 64 & $A_{1,2}^{6}B_{2,3}B_{4,7}$ & X (Jac) \\
    171 & 64 & $A_{2,3}A_{3,4}B_{2,3}^{2}B_{3,5}$ & X (Jac) \\
    172 & 64 & $A_{4,5}B_{2,3}^{4}$ & X (Jac) \\
    173 & 64 & $A_{1,2}^{4}B_{2,3}B_{3,5}^2$ & X (Jac) \\
    174 & 64 & $A_{1,2}B_{2,3}^{4}B_{3,5}$ & X (Jac) \\
    175 & 64 & $A_{1,2}A_{2,3}^2A_{4,5}C_{3,4}$ & X (Jac) \\
    176 & 64 & $A_{1,2}A_{3,4}^2B_{2,3}C_{3,4}$ & X (dim)\\
    177 & 64 & $A_{1,2}^{3}A_{4,5}B_{2,3}C_{3,4}$ & X (Jac) \\
    178 & 64 & $A_{1,2}^{2}A_{2,3}^{2}B_{3,5}C_{3,4}$ & X (Jac) \\
    179 & 64 & $A_{2,3}A_{3,4}B_{2,3}^{2}C_{3,4}$ & X (dim)\\
    180 & 64 & $A_{1,2}^{4}B_{2,3}B_{3,5}C_{3,4}$ & X (Jac) \\
    181 & 64 & $A_{1,2}B_{2,3}^{4}C_{3,4}$ & X (dim)\\
    182 & 64 & $A_{1,2}^{4}A_{2,3}^{2}C_{4,5}$ & X (Jac) \\
    183 & 64 & $A_{1,2}^{6}B_{2,3}C_{4,5}$ & X (Jac) \\
    184 & 64 & $A_{1,2}^{2}A_{2,3}^{2}C_{3,4}^{2}$ & X (Jac) \\
    185 & 64 & $A_{1,2}^{4}B_{2,3}C_{3,4}^{2}$ & X (Jac) \\
    186 & 64 & $A_{1,2}^{2}A_{3,4}^{2}D_{4,6}$ & X (dim)\\
    187 & 64 & $A_{1,2}^{4}A_{4,5}D_{4,6}$ & X (Jac) \\
    188 & 64 & $A_{1,2}A_{2,3}A_{3,4}B_{2,3}D_{4,6}$ & X (dim)\\
    189 & 64 & $A_{1,2}^{5}B_{3,5}D_{4,6}$ & X (Jac) \\
    190 & 64 & $A_{2,3}^{2}B_{2,3}^{2}D_{4,6}$ & X (dim)\\
    191 & 64 & $A_{1,2}^{2}B_{2,3}^{3}D_{4,6}$ & X (dim)\\
    192 & 64 & $A_{1,2}^{5}C_{3,4}D_{4,6}$ & X (Jac) \\
    193 & 64 & $A_{1,2}A_{2,3}A_{3,4}A_{4,5}G_{2,4}$ & X (Jac) \\
    194 & 64 & $A_{1,2}^{5}A_{5,6}G_{2,4}$ & X (Jac) \\
    195 & 64 & $A_{2,3}^{2}A_{4,5}B_{2,3}G_{2,4}$ & X (Jac) \\
    196 & 64 & $A_{1,2}^{2}A_{2,3}A_{3,4}B_{3,5}G_{2,4}$ & X (Jac) \\ 
    197 & 64 & $A_{3,4}^{2}B_{2,3}^{2}G_{2,4}$ & X (dim)\\
    198 & 64 & $A_{1,2}^{2}A_{4,5}B_{2,3}^{2}G_{2,4}$ & X (Jac) \\
    199 & 64 & $A_{1,2}A_{2,3}^2B_{2,3}B_{3,5}G_{2,4}$ & X (Jac) \\
    200 & 64 & $A_{1,2}^{3}B_{2,3}^{2}B_{3,5}G_{2,4}$ & X (Jac) \\
\hline 
\end{tabular}
\end{center}

\begin{center}
\begin{tabular}{|p{0.7cm}p{1.3cm}p{3cm}p{1.5cm}|}
    \hline
    \# & $\dim V_1$ & Root system & Result\\
    \hline
    \hline 
    201 & 64 & $B_{2,3}^{5}G_{2,4}$ & X (dim)\\
    202 & 64 & $A_{1,2}^{2}A_{2,3}A_{3,4}C_{3,4}G_{2,4}$ & X (Jac) \\
    203 & 64 & $A_{1,2}A_{2,3}^2B_{2,3}C_{3,4}G_{2,4}$ & X (Jac) \\
    204 & 64 & $A_{1,2}^{3}B_{2,3}^{2}C_{3,4}G_{2,4}$ & X (Jac) \\
    205 & 64 & $A_{1,2}^{2}A_{2,3}^{2}D_{4,6}G_{2,4}$ & X (Jac) \\
    206 & 64 & $A_{1,2}^{4}B_{2,3}D_{4,6}G_{2,4}$ & X (Jac) \\
    207 & 64 & $A_{1,2}^{2}A_{3,4}^{2}G_{2,4}^{2}$ & X (Jac) \\
    208 & 64 & $A_{1,2}^{4}A_{4,5}G_{2,4}^{2}$ & X (Jac) \\
    209 & 64 & $A_{1,2}A_{2,3}A_{3,4}B_{2,3}G_{2,4}^{2}$ & X (Jac) \\
    210 & 64 & $A_{1,2}^{5}B_{3,5}G_{2,4}^{2}$ & X (Jac) \\
    211 & 64 & $A_{2,3}^{2}B_{2,3}^{2}G_{2,4}^{2}$ & X (Jac) \\
    212 & 64 & $A_{1,2}^{2}B_{2,3}^{3}G_{2,4}^{2}$ & X (Jac) \\
    213 & 64 & $A_{1,2}^{5}C_{3,4}G_{2,4}^{2}$ & X (Jac) \\
    214 & 64 & $A_{1,2}^{2}A_{2,3}^{2}G_{2,4}^{3}$ & X (Jac) \\
    215 & 64 & $A_{1,2}^{4}B_{2,3}G_{2,4}^{3}$ & X (Jac) \\
    216 & 64 & $A_{1,2}A_{2,3}^2A_{3,4}^3$ & X (dim)\\
    217 & 64 & $A_{2,3}^{5}A_{4,5}$ & X (Jac) \\
    218 & 64 & $A_{1,2}^{3}A_{2,3}^{2}A_{3,4}A_{4,5}$ & X (Jac) \\
    219 & 64 & $A_{1,2}^{7}A_{2,3}A_{5,6}$ & X (Jac) \\
    220 & 64 & $A_{2,3}^{3}A_{3,4}^{2}B_{2,3}$ & X (dim)\\
    221 & 64 & $A_{1,2}^{3}A_{3,4}^{3}B_{2,3}$ & X (dim)\\
    222 & 64 & $A_{1,2}^{2}A_{2,3}^{3}A_{4,5}B_{2,3}$ & X (Jac) \\
    223 & 64 & $A_{1,2}^{5}A_{3,4}A_{4,5}B_{2,3}$ & X (Jac) \\
    224 & 64 & $A_{1,2}A_{2,3}^5B_{3,5}$ & X (Jac) \\
    225 & 64 & $A_{1,2}^{4}A_{2,3}^{2}A_{3,4}B_{3,5}$ & X (Jac) \\
    226 & 64 & $A_{1,2}^{2}A_{2,3}A_{3,4}^{2}B_{2,3}^{2}$ & X (dim)\\
    227 & 64 & $A_{1,2}^{4}A_{2,3}A_{4,5}B_{2,3}^{2}$ & X (Jac) \\
    228 & 64 & $A_{1,2}^{3}A_{2,3}^{3}B_{2,3}B_{3,5}$ & X (Jac) \\
    229 & 64 & $A_{1,2}^{6}A_{3,4}B_{2,3}B_{3,5}$ & X (Jac) \\
    230 & 64 & $A_{1,2}A_{2,3}^2A_{3,4}B_{2,3}^{3}$ & X (dim)\\
    231 & 64 & $A_{1,2}^{5}A_{2,3}B_{2,3}^{2}B_{3,5}$ & X (Jac) \\
    232 & 64 & $A_{2,3}^{3}B_{2,3}^{4}$ & X (dim)\\
    233 & 64 & $A_{1,2}^{3}A_{3,4}B_{2,3}^{4}$ & X (dim)\\
    234 & 64 & $A_{1,2}^{2}A_{2,3}B_{2,3}^{5}$ & X (dim)\\
    235 & 64 & $A_{1,2}A_{2,3}^5C_{3,4}$ & X (Jac) \\
    236 & 64 & $A_{1,2}^{4}A_{2,3}^{2}A_{3,4}C_{3,4}$ & X (Jac) \\
    237 & 64 & $A_{1,2}^{3}A_{2,3}^{3}B_{2,3}C_{3,4}$ & X (Jac) \\
    238 & 64 & $A_{1,2}^{6}A_{3,4}B_{2,3}C_{3,4}$ & X (Jac) \\
    239 & 64 & $A_{1,2}^{5}A_{2,3}B_{2,3}^{2}C_{3,4}$ & X (Jac) \\ 
    240 & 64 & $A_{1,2}^{4}A_{2,3}^{3}D_{4,6}$ & X (Jac) \\
    \hline
\end{tabular}
\end{center}
\end{multicols}

\begin{multicols}{2}
\begin{center}
\begin{tabular}{|p{0.7cm}p{1.3cm}p{3cm}p{1.5cm}|}
    \hline
    \# & $\dim V_1$ & Root system & Result\\
    \hline
    \hline 
    241 & 64 & $A_{1,2}^{7}A_{3,4}D_{4,6}$ & X (Jac) \\
    242 & 64 & $A_{1,2}^{6}A_{2,3}B_{2,3}D_{4,6}$ & X (Jac) \\
    243 & 64 & $A_{1,2}A_{2,3}^4A_{3,4}G_{2,4}$ & X (Jac) \\
    244 & 64 & $A_{1,2}^{4}A_{2,3}A_{3,4}^{2}G_{2,4}$ & X (Jac) \\
    245 & 64 & $A_{1,2}^{6}A_{2,3}A_{4,5}G_{2,4}$ & X (Jac) \\
    246 & 64 & $A_{2,3}^{5}B_{2,3}G_{2,4}$ & X (Jac) \\
    247 & 64 & $A_{1,2}^{3}A_{2,3}^{2}A_{3,4}B_{2,3}G_{2,4}$ & X (Jac) \\
    248 & 64 & $A_{1,2}^{7}A_{2,3}B_{3,5}G_{2,4}$ & X (Jac) \\
    249 & 64 & $A_{1,2}^{2}A_{2,3}^{3}B_{2,3}^{2}G_{2,4}$ & X (Jac) \\
    250 & 64 & $A_{1,2}^{5}A_{3,4}B_{2,3}^{2}G_{2,4}$ & X (Jac) \\
    251 & 64 & $A_{1,2}^{4}A_{2,3}B_{2,3}^{3}G_{2,4}$ & X (Jac) \\
    252 & 64 & $A_{1,2}^{7}A_{2,3}C_{3,4}G_{2,4}$ & X (Jac) \\
    253 & 64 & $A_{1,2}^{4}A_{2,3}^{3}G_{2,4}^{2}$ & X (Jac) \\
    254 & 64 & $A_{1,2}^{7}A_{3,4}G_{2,4}^{2}$ & X (Jac) \\
    255 & 64 & $A_{1,2}^{6}A_{2,3}B_{2,3}G_{2,4}^{2}$ & X (Jac) \\
    256 & 64 & $A_{2,3}^{8}$ & X (char)\\
    257 & 64 & $A_{1,2}^{3}A_{2,3}^{5}A_{3,4}$ & X (Jac) \\
    258 & 64 & $A_{1,2}^{6}A_{2,3}^{2}A_{3,4}^{2}$ & X (Jac) \\
    259 & 64 & $A_{1,2}^{8}A_{2,3}^{2}A_{4,5}$ & X (Jac) \\
    260 & 64 & $A_{1,2}^{2}A_{2,3}^{6}B_{2,3}$ & X (char)\\
    261 & 64 & $A_{1,2}^{5}A_{2,3}^{3}A_{3,4}B_{2,3}$ & X (Jac) \\
    262 & 64 & $A_{1,2}^{8}A_{3,4}^{2}B_{2,3}$ & X (char)\\
    263 & 64 & $A_{1,2}^{10}A_{4,5}B_{2,3}$ & X (Jac) \\
    264 & 64 & $A_{1,2}^{9}A_{2,3}^{2}B_{3,5}$ & X (Jac) \\
    265 & 64 & $A_{1,2}^{4}A_{2,3}^{4}B_{2,3}^{2}$ & X (char)\\
    266 & 64 & $A_{1,2}^{7}A_{2,3}A_{3,4}B_{2,3}^{2}$ & X (Jac) \\
    267 & 64 & $A_{1,2}^{11}B_{2,3}B_{3,5}$ & X (Jac) \\
    268 & 64 & $A_{1,2}^{6}A_{2,3}^{2}B_{2,3}^{3}$ & X (Jac) \\
    269 & 64 & $A_{1,2}^{8}B_{2,3}^{4}$ & X (char)\\
    270 & 64 & $A_{1,2}^{9}A_{2,3}^{2}C_{3,4}$ & X (Jac) \\
    271 & 64 & $A_{1,2}^{11}B_{2,3}C_{3,4}$ & X (Jac) \\
    272 & 64 & $A_{1,2}^{12}D_{4,6}$ & X (Jac) \\
    273 & 64 & $A_{1,2}^{6}A_{2,3}^{4}G_{2,4}$ & X (Jac) \\
    274 & 64 & $A_{1,2}^{9}A_{2,3}A_{3,4}G_{2,4}$ & X (Jac) \\
    275 & 64 & $A_{1,2}^{8}A_{2,3}^{2}B_{2,3}G_{2,4}$ & X (Jac) \\
    276 & 64 & $A_{1,2}^{10}B_{2,3}^{2}G_{2,4}$ & X (Jac) \\
    277 & 64 & $A_{1,2}^{12}G_{2,4}^{2}$ & X (char)\\
    278 & 64 & $A_{1,2}^{8}A_{2,3}^{5}$ & X (char)\\
    279 & 64 & $A_{1,2}^{11}A_{2,3}^{2}A_{3,4}$ & X (char)\\
    280 & 64 & $A_{1,2}^{10}A_{2,3}^{3}B_{2,3}$ & X (Jac) \\

\hline 
    \end{tabular}
\end{center}

\begin{center}
    \begin{tabular}{|p{0.7cm}p{1.3cm}p{3cm}p{1.5cm}|}
    \hline
    \# & $\dim V_1$ & Root system & Result\\
    \hline
    \hline 
    281 & 64 & $A_{1,2}^{13}A_{3,4}B_{2,3}$ & X (char)\\
    282 & 64 & $A_{1,2}^{12}A_{2,3}B_{2,3}^{2}$ & X (char)\\
    283 & 64 & $A_{1,2}^{14}A_{2,3}G_{2,4}$ & X (char)\\
    284 & 64 & $A_{1,2}^{16}A_{2,3}^{2}$ & X (char)\\
    285 & 64 & $A_{1,2}^{18}B_{2,3}$ & X (char)\\
\hline
    286 & 72 & $C_{4,4}^{2}$ & \Checkmark \\
    287 & 72 & $A_{4,4}^{3}$ & X (dim)\\
\hline
    288 & 80 & $A_{8,6}$ & X (dim)\\
    289 & 80 & $D_{4,4}F_{4,6}$ & X (Jac) \\
    290 & 80 & $A_{2,2}B_{2,2}^{2}F_{4,6}$ & X (Jac) \\
    291 & 80 & $A_{5,4}^{2}B_{2,2}$ & X (Jac) \\
    292 & 80 & $A_{2,2}^{3}D_{4,4}^{2}$ & X (Jac) \\
    293 & 80 & $B_{2,2}^{8}$  & \Checkmark (DGM) \\
    294 & 80 & $A_{2,2}^{4}B_{2,2}^{2}D_{4,4}$ & X (Jac) \\
    295 & 80 & $A_{2,2}^{5}B_{2,2}^{4}$ & X (Jac) \\
    296 & 80 & $A_{2,2}^{10}$ & X (dim) \\
\hline
    297 & 96 & $A_{1,1}A_{3,2}E_{6,6}$ & X (dim) \\
    298 & 96 & $A_{1,1}^{2}A_{5,3}C_{5,3}$ & X (Jac) \\
    299 & 96 & $A_{3,2}^{2}D_{6,5}$ & X (Jac) \\
    300 & 96 & $A_{1,1}^{3}C_{3,2}D_{6,5}$ & X (Jac) \\
    301 & 96 & $A_{1,1}^{2}D_{5,4}^{2}$ & X (Jac) \\
    302 & 96 & $A_{1,1}^{2}C_{3,2}C_{5,3}G_{2,2}$ & X (Jac)\\
    303 & 96 & $A_{1,1}^{6}E_{6,6}$ & X (Jac) \\
    304 & 96 & $A_{1,1}A_{3,2}^2A_{7,4}$ & X (Jac) \\
    305 & 96 & $A_{1,1}^{4}A_{7,4}C_{3,2}$ & X (Jac) \\
    306 & 96 & $A_{3,2}^{2}C_{3,2}D_{5,4}$ & X (Jac) \\
    307 & 96 & $A_{1,1}^{3}C_{3,2}^{2}D_{5,4}$ & X (Jac) \\
    308 & 96 & $A_{1,1}^{5}A_{3,2}D_{6,5}$ & X (Jac) \\
    309 & 96 & $A_{1,1}^{4}A_{3,2}C_{5,3}G_{2,2}$ & X (Jac) \\
    310 & 96 & $A_{1,1}^{3}D_{4,3}D_{5,4}G_{2,2}$  & X (Jac) \\
    311 & 96 & $A_{1,1}^{3}D_{5,4}G_{2,2}^{3}$  & X (Jac) \\
    312 & 96 & $A_{1,1}^{6}A_{3,2}A_{7,4}$ & X (Jac) \\
    313 & 96 & $A_{1,1}A_{3,2}^2C_{3,2}^{3}$ & X (Jac) \\
    314 & 96 & $A_{1,1}^{4}C_{3,2}^{4}$ & X (Jac) \\
    315 & 96 & $A_{1,1}A_{3,2}^2A_{5,3}D_{4,3}$ & X (Jac) \\
    316 & 96 & $A_{1,1}^{4}A_{5,3}C_{3,2}D_{4,3}$ & X (Jac) \\
    317 & 96 & $A_{1,1}^{2}A_{3,2}^{3}D_{5,4}$ & X (Jac) \\
    318 & 96 & $A_{1,1}^{5}A_{3,2}C_{3,2}D_{5,4}$ & X (Jac) \\
    319 & 96 & $A_{1,1}^{10}D_{6,5}$  & X (Jac) \\
    320 & 96 & $A_{1,1}^{4}D_{4,3}^{3}$ & X (dim) \\
\hline 
\end{tabular}
\end{center}
\end{multicols}

\begin{multicols}{2}
\begin{center}
\begin{tabular}{|p{0.7cm}p{1.3cm}p{3cm}p{1.5cm}|}
    \hline
    \# & $\dim V_1$ & Root system & Result\\
    \hline
    \hline 
    321 & 96 & $A_{1,1}^{4}A_{5,3}^{2}G_{2,2}$ & X (dim) \\
    322 & 96 & $A_{1,1}^{9}C_{5,3}G_{2,2}$ & X (Jac) \\
    323 & 96 & $A_{1,1}A_{3,2}^2C_{3,2}D_{4,3}G_{2,2}$ & X (Jac) \\
    324 & 96 & $A_{1,1}^{4}C_{3,2}^{2}D_{4,3}G_{2,2}$ & X (Jac) \\
    325 & 96 & $A_{1,1}A_{3,2}^2A_{5,3}G_{2,2}^{2}$ & X (Jac) \\
    326 & 96 & $A_{1,1}^{4}A_{5,3}C_{3,2}G_{2,2}^{2}$ & X (Jac) \\
    327 & 96 & $A_{1,1}^{4}D_{4,3}^{2}G_{2,2}^{2}$ & X (dim) \\
    328 & 96 & $A_{1,1}A_{3,2}^2C_{3,2}G_{2,2}^{3}$ & X (Jac) \\
    329 & 96 & $A_{1,1}^{4}C_{3,2}^{2}G_{2,2}^{3}$ & X (Jac) \\
    330 & 96 & $A_{1,1}^{4}D_{4,3}G_{2,2}^{4}$ & X (dim) \\
    331 & 96 & $A_{1,1}^{4}G_{2,2}^{6}$ & X (dim) \\
    332 & 96 & $A_{1,1}^{11}A_{7,4}$ & X (Jac) \\
    333 & 96 & $A_{3,2}^{5}C_{3,2}$ & X (Jac) \\
    334 & 96 & $A_{1,1}^{3}A_{3,2}^{3}C_{3,2}^{2}$ & X (Jac) \\
    335 & 96 & $A_{1,1}^{6}A_{3,2}C_{3,2}^{3}$ & X (Jac) \\
    336 & 96 & $A_{1,1}^{6}A_{3,2}A_{5,3}D_{4,3}$ & X (Jac) \\
    337 & 96 & $A_{1,1}^{7}A_{3,2}^{2}D_{5,4}$ & X (Jac) \\
    338 & 96 & $A_{1,1}^{10}C_{3,2}D_{5,4}$ & X (Jac) \\
    339 & 96 & $A_{1,1}^{3}A_{3,2}^{3}D_{4,3}G_{2,2}$ & X (Jac) \\
    340 & 96 & $A_{1,1}^{6}A_{3,2}C_{3,2}D_{4,3}G_{2,2}$ & X (Jac) \\
    341 & 96 & $A_{1,1}^{6}A_{3,2}A_{5,3}G_{2,2}^{2}$ & X (Jac) \\
    342 & 96 & $A_{1,1}^{3}A_{3,2}^{3}G_{2,2}^{3}$ & X (Jac) \\
    343 & 96 & $A_{1,1}^{6}A_{3,2}C_{3,2}G_{2,2}^{3}$ & X (Jac) \\
    344 & 96 & $A_{1,1}^{2}A_{3,2}^{6}$ & X (Jac) \\
    345 & 96 & $A_{1,1}^{5}A_{3,2}^{4}C_{3,2}$ & X (Jac) \\
    346 & 96 & $A_{1,1}^{8}A_{3,2}^{2}C_{3,2}^{2}$ & X (Jac) \\
    347 & 96 & $A_{1,1}^{11}C_{3,2}^{3}$ & X (Jac) \\
    348 & 96 & $A_{1,1}^{11}A_{5,3}D_{4,3}$ & X (Jac) \\
    349 & 96 & $A_{1,1}^{12}A_{3,2}D_{5,4}$ & X (Jac) \\
    350 & 96 & $A_{1,1}^{8}A_{3,2}^{2}D_{4,3}G_{2,2}$ & X (Jac) \\
    351 & 96 & $A_{1,1}^{11}C_{3,2}D_{4,3}G_{2,2}$ & X (dim)\\
    352 & 96 & $A_{1,1}^{11}A_{5,3}G_{2,2}^{2}$ & X (Jac) \\
    353 & 96 & $A_{1,1}^{8}A_{3,2}^{2}G_{2,2}^{3}$ & X (Jac) \\ 354 & 96 & $A_{1,1}^{11}C_{3,2}G_{2,2}^{3}$ & X (dim)\\
    355 & 96 & $A_{1,1}^{7}A_{3,2}^{5}$ & X (Jac) \\
    356 & 96 & $A_{1,1}^{10}A_{3,2}^{3}C_{3,2}$ & X (Jac) \\
    357 & 96 & $A_{1,1}^{13}A_{3,2}C_{3,2}^{2}$ & X (Jac) \\
    358 & 96 & $A_{1,1}^{17}D_{5,4}$ & X (char)\\
    359 & 96 & $A_{1,1}^{13}A_{3,2}D_{4,3}G_{2,2}$ & X (dim) \\
    360 & 96 & $A_{1,1}^{13}A_{3,2}G_{2,2}^{3}$ & X (dim) \\
\hline
    \end{tabular}
\end{center}

\begin{center}
 \begin{tabular}{|p{0.7cm}p{1.3cm}p{3cm}p{1.5cm}|}
    \hline
    \# & $\dim V_1$ & Root system & Result\\
    \hline
    \hline 
    361 & 96 & $A_{1,1}^{12}A_{3,2}^{4}$ & \Checkmark\\
    362 & 96 & $A_{1,1}^{15}A_{3,2}^{2}C_{3,2}$ & X (char)\\
    363 & 96 & $A_{1,1}^{18}C_{3,2}^{2}$  & X (dim) \\
    364 & 96 & $A_{1,1}^{18}D_{4,3}G_{2,2}$ & X (dim) \\
    365 & 96 & $A_{1,1}^{18}G_{2,2}^{3}$ & X (dim) \\
    366 & 96 & $A_{1,1}^{17}A_{3,2}^{3}$ & X (char) \\
    367 & 96 & $A_{1,1}^{20}A_{3,2}C_{3,2}$ & X (dim) \\
    368 & 96 & $A_{1,1}^{22}A_{3,2}^{2}$ & X (dim) \\
    369 & 96 & $A_{1,1}^{25}C_{3,2}$ & X (dim) \\
    370 & 96 & $A_{1,1}^{27}A_{3,2}$ & X (dim) \\
    371 & 96 & $A_{1,1}^{32}$ & \Checkmark (lat) \\
\hline
    372 & 104 & $F_{4,4}^{2}$ & \Checkmark \\
\hline
    373 & 120 & $A_{10,4}$  & X (dim)\\
\hline
    374 & 128 & $A_{2,1}B_{2,1}B_{5,3}^2$ & X (dim)\\
    375 & 128 & $A_{2,1}B_{2,1}B_{5,3}C_{5,2}$ & X (dim)\\
    376 & 128 & $A_{2,1}B_{2,1}C_{5,2}^{2}$  & X (dim)\\
    377 & 128 & $A_{2,1}^{3}F_{4,3}^{2}$ & X (dim)\\
    378 & 128 & $A_{8,3}B_{2,1}^{2}D_{4,2}$  & X (dim)\\
    379 & 128 & $A_{5,2}B_{2,1}B_{5,3}D_{4,2}$ & X (Jac) \\
    380 & 128 & $A_{5,2}B_{2,1}C_{5,2}D_{4,2}$ & X (Jac) \\
    381 & 128 & $B_{2,1}^{2}D_{4,2}^{2}F_{4,3}$  & X (Jac) \\
    382 & 128 & $B_{2,1}^{5}E_{6,4}$ & X (dim)\\
    383 & 128 & $A_{2,1}A_{8,3}B_{2,1}^{4}$ & X (Jac) \\
    384 & 128 & $A_{2,1}A_{5,2}B_{2,1}^{3}B_{5,3}$ & X (Jac) \\
    385 & 128 & $A_{2,1}A_{5,2}B_{2,1}^{3}C_{5,2}$ & X (dim)\\
    386 & 128 & $A_{2,1}B_{2,1}^{4}D_{4,2}F_{4,3}$ & X (Jac) \\
    387 & 128 & $A_{2,1}^{5}B_{2,1}E_{6,4}$ & X (Jac) \\
    388 & 128 & $A_{2,1}^{6}A_{8,3}$ & X (Jac) \\
    389 & 128 & $A_{5,2}^{2}B_{2,1}^{3}D_{4,2}$ & X (dim)\\
    390 & 128 & $A_{2,1}^{2}D_{4,2}^{4}$ & X (dim)\\
    391 & 128 & $A_{2,1}^{2}B_{2,1}^{6}F_{4,3}$ & X (Jac) \\
    392 & 128 & $A_{2,1}^{6}D_{4,2}F_{4,3}$ & X (Jac) \\
    393 & 128 & $A_{2,1}A_{5,2}^2B_{2,1}^{5}$ & X (Jac) \\
    394 & 128 & $A_{2,1}^{3}B_{2,1}^{2}D_{4,2}^{3}$ & X (Jac) \\
    395 & 128 & $A_{2,1}^{7}B_{2,1}^{2}F_{4,3}$ & X (Jac) \\
    396 & 128 & $A_{2,1}^{6}A_{5,2}^{2}B_{2,1}$ & X (Jac) \\
    397 & 128 & $B_{2,1}^{10}D_{4,2}$ & X (Jac) \\
    398 & 128 & $A_{2,1}^{4}B_{2,1}^{4}D_{4,2}^{2}$ & X (Jac) \\
    399 & 128 & $A_{2,1}B_{2,1}^{12}$ & X (dim)\\
    400 & 128 & $A_{2,1}^{5}B_{2,1}^{6}D_{4,2}$ & X (Jac) \\
\hline 
    \end{tabular}
\end{center}
\end{multicols}

\begin{multicols}{2}
\begin{center}
  \begin{tabular}{|p{0.7cm}p{1.3cm}p{3cm}p{1.5cm}|}
    \hline
    \# & $\dim V_1$ & Root system & Result\\
    \hline
    \hline 
    401 & 128 & $A_{2,1}^{9}D_{4,2}^{2}$ & X (Jac) \\
    402 & 128 & $A_{2,1}^{6}B_{2,1}^{8}$ & X (Jac) \\
    403 & 128 & $A_{2,1}^{10}B_{2,1}^{2}D_{4,2}$ & X (dim)\\
    404 & 128 & $A_{2,1}^{11}B_{2,1}^{4}$ & X (char)\\
    405 & 128 & $A_{2,1}^{16}$ & \Checkmark (lat) \\
\hline
    406 & 144 & $B_{4,2}^{4}$ & \Checkmark (DGM)\\
    407 & 144 & $A_{6,2}^{3}$ & X (dim)\\
\hline
    408 & 160 & $C_{3,1}^{2}D_{5,2}^{2}G_{2,1}^{2}$ & X (Jac) \\
    409 & 160 & $D_{5,2}^{2}G_{2,1}^{5}$ & X (Jac) \\
    410 & 160 & $A_{3,1}^{3}C_{3,1}^{2}D_{5,2}G_{2,1}^{2}$ & X (Jac) \\
    411 & 160 & $A_{3,1}^{3}D_{5,2}G_{2,1}^{5}$ & X (Jac) \\
    412 & 160 & $A_{3,1}^{6}C_{3,1}^{2}G_{2,1}^{2}$ & X (dim)\\
    413 & 160 & $A_{3,1}^{6}G_{2,1}^{5}$ & X (dim)\\
\hline
    414 & 192 & $B_{3,1}C_{9,2}$ & X (dim)\\
    415 & 192 & $A_{9,2}B_{3,1}C_{4,1}^{2}$ & X (Jac) \\
    416 & 192 & $A_{4,1}C_{4,1}D_{6,2}^{2}$ & X (dim)\\
    417 & 192 & $A_{4,1}^{3}A_{9,2}B_{3,1}$ & X (Jac) \\
    418 & 192 & $B_{3,1}^{4}C_{4,1}^{3}$ & X (dim)\\
    419 & 192 & $A_{4,1}^{2}C_{4,1}^{4}$ & X (dim)\\
    420 & 192 & $B_{3,1}^{6}D_{6,2}$ & X (Jac) \\
    421 & 192 & $A_{4,1}^{2}B_{3,1}^{2}C_{4,1}D_{6,2}$ & X (Jac) \\
    422 & 192 & $A_{4,1}B_{3,1}^{8}$ & X (dim)\\
    423 & 192 & $A_{4,1}^{3}B_{3,1}^{4}C_{4,1}$ & X (Jac) \\
    424 & 192 & $A_{4,1}^{5}C_{4,1}^{2}$ & X (Jac) \\
    425 & 192 & $A_{4,1}^{8}$ & \Checkmark (lat)\\
\hline
    426 & 224 & $D_{7,2}E_{7,3}$ & X (dim)\\
    427 & 224 & $C_{5,1}D_{7,2}E_{6,2}$ & X (dim)\\
    428 & 224 & $A_{5,1}D_{4,1}^{2}E_{7,3}$ & X (dim)\\
    429 & 224 & $A_{5,1}C_{5,1}D_{4,1}^{2}E_{6,2}$ & X (dim)\\
    430 & 224 & $A_{5,1}^{3}D_{4,1}D_{7,2}$ & X (dim)\\
    431 & 224 & $A_{5,1}^{4}D_{4,1}^{3}$  & X (dim)\\
    432 & 224 & $D_{4,1}^{8}$ & \Checkmark (lat)\\
\hline
    433 & 272 & $B_{8,2}^{2}$ & \Checkmark (DGM) \\
\hline
    434 & 288 & $D_{5,1}^{3}D_{9,2}$ & X (dim)\\
    435 & 288 & $A_{7,1}D_{5,1}^{5}$ & X (dim) \\
\hline
    436 & 320 & $A_{8,1}C_{8,1}F_{4,1}^{2}$ & X (dim)\\
    437 & 320 & $A_{8,1}B_{5,1}E_{7,2}F_{4,1}$ & X (dim)\\
    438 & 320 & $A_{8,1}^{4}$ & \Checkmark (lat) \\
\hline
    439 & 416 & $D_{7,1}D_{13,2}$ & X (dim)\\
    440 & 416 & $A_{11,1}D_{7,1}^{3}$ & X (dim)\\
\hline
\end{tabular}
\end{center}

\begin{center}
\begin{tabular}{|p{0.7cm}p{1.3cm}p{3cm}p{1.5cm}|}
    \hline
    \# & $\dim V_1$ & Root system & Result\\
    \hline
    \hline 
    441 & 416 & $D_{7,1}^{2}E_{6,1}^{3}$  & X (dim)\\
\hline
    442 & 480 & $D_{8,1}^{4}$ & \Checkmark (lat)\\
\hline
    443 & 528 & $B_{16,2}$ & X (dim)\\
\hline
    444 & 576 & $A_{16,1}^{2}$ & \Checkmark (lat)\\
\hline 
    445 & 992 & $D_{16,1}^{2}$ & \Checkmark (lat)\\
    446 & 992 & $D_{16,1}E_{8,1}^{2}$ & \Checkmark (lat)\\
    447 & 992 & $E_{8,1}^{4}$ & \Checkmark (lat) \\
\hline
    448 & 1088 & $A_{32,1}$ & X (dim) \\
\hline
    449 & 2016 & $D_{32,1}$ & \Checkmark (lat)\\
    \hline 
    \end{tabular}
\end{center}
    \end{multicols}

\end{document}